\documentclass[12pt,a4paper,reqno]{amsart}

\makeatletter
\let\@wraptoccontribs\wraptoccontribs
\makeatother

\usepackage{amssymb}
\usepackage{amscd}
\usepackage{enumerate}
\usepackage{graphicx}
\usepackage{siunitx}
\usepackage{tikz-cd}
\usepackage{bm}
\usepackage{mathtools}

\theoremstyle{plain}

\newtheorem{theorem}{Theorem}[section]
\newtheorem{proposition}[theorem]{Proposition}
\newtheorem{lemma}[theorem]{Lemma}
\newtheorem{corollary}[theorem]{Corollary}

\theoremstyle{definition}

\newtheorem{definition}[theorem]{Definition}
\newtheorem{remark}[theorem]{Remark}

\newtheorem{example}[theorem]{Example}

\numberwithin{equation}{section}

\newcommand\E{\mathbb{E}}

\newcommand\R{\mathbb{R}}
\newcommand\C{\mathbb{C}}

\newcommand\N{\mathbb{N}}

\newcommand\A{\mathcal{A}}
\newcommand\semicirc{{\mu_{\mathrm{sc}}}}
\newcommand\eps{\varepsilon}

\DeclareMathOperator{\Log}{Log}
\DeclareMathOperator{\cosec}{cosec}
\DeclareMathOperator{\tr}{tr}
\DeclareMathOperator{\Tr}{Tr}
\DeclareMathOperator{\sa}{sa}
\DeclareMathOperator{\Var}{Var}
\DeclareMathOperator{\diag}{diag}
\DeclareMathOperator{\op}{op}

\DeclarePairedDelimiter{\ip}{\langle}{\rangle}
\DeclarePairedDelimiter{\norm}{\lVert}{\rVert}

\begin{document}

\title{Fractional free convolution powers}

\author{Dimitri Shlyakhtenko}
\address{University of California, Los Angeles, Department of Mathematics, Los Angeles, CA 90095-1555.}
\email{shlyakht@math.ucla.edu}

\author{Terence Tao}
\address{University of California, Los Angeles, Department of Mathematics, Los Angeles, CA 90095-1555.}
\email{tao@math.ucla.edu}

\contrib[with an appendix by]{David Jekel}
\address{University of California, San Diego, Department of Mathematics, La Jolla, CA 92093-0112}
\email{djekel@ucsd.edu}


\begin{abstract}
The extension $k \mapsto \mu^{\boxplus k}$ of the concept of a free convolution power to the case of non-integer $k \geq 1$ was introduced by Bercovici-Voiculescu and Nica-Speicher, and related to the minor process in random matrix theory.  In this paper we give two proofs of the monotonicity of the free entropy and free Fisher information of the (normalized) free convolution power in this continuous setting, and also establish an intriguing variational description of this process.
\end{abstract}

\maketitle


\section{Introduction}

\subsection{Integer free convolution powers}

In this paper we assume familiarity with noncommutative probability, particularly the concept of free independence (see, e.g., \cite{vdn}). 

In \cite{voi}, Voiculescu introduced the notion of the \emph{free convolution} $\mu \boxplus \nu$ of two compactly supported probability measures $\mu,\nu$ on $\R$.  There are multiple ways to define this operation.  One is to define $\mu \boxplus \nu$ to be the law of $X+Y$, where $X, Y$ are freely independent (real) noncommutative random variables with law $\mu,\nu$ respectively.  Another is to define $\mu \boxplus \nu$ to be the asymptotic empirical spectral distribution of $A+B$ as $N \to \infty$, where $A,B$ are classically independent bounded $N \times N$ random Hermitian matrices, each invariant under unitary conjugation, and whose empirical spectral distribution converges to $\mu,\nu$ respectively.  A third way is to introduce the \emph{Cauchy transform}\footnote{One can also write $G_\mu = -s_\mu$, where $s_\mu(z) \coloneqq \int_\R \frac{d\mu(x)}{x-z}$ is the \emph{Stieltjes transform} of $\mu$; however it will be slightly more convenient to work with the Cauchy transform instead of the Stieltjes transform to reduce the number of minus signs in our formulae.} $G_\mu \colon \C \backslash \mathrm{supp}(\mu) \to \C$ of a compactly supported probability measure $\mu$ by the formula
\begin{equation}\label{gdef}
 G_\mu(z) \coloneqq \int_\R \frac{d\mu(x)}{z-x}
\end{equation}
for $z \in \C \backslash \mathrm{supp}(\mu)$ (in particular one has $G_\mu(z) = \frac{1}{z} + O(\frac{1}{|z|^2})$ as $|z| \to \infty$), and then define the \emph{$R$-transform} $R_\mu(s)$ for sufficiently small complex numbers $s$ by requiring that 
\begin{equation}\label{rg}
 \frac{1}{G_\mu(z)} + R_\mu( G_\mu(z) ) = z
\end{equation}
for all sufficiently large $z$.  For sufficiently small $z$ one has the convergent Taylor expansion
$$ R_\mu(s) = \sum_{n=0}^\infty \kappa_{n+1}(\mu) s^n$$
where 
\begin{align*}
\kappa_1(\mu) &= \int_\R x\ d\mu \\
\kappa_2(\mu) &= \int_\R x^2\ d\mu - \left(\int_\R x\ d\mu\right)^2\\
\kappa_3(\mu) &= \int_\R x^3\ d\mu - 3 \left(\int_\R x\ d\mu\right) \left(\int_\R x^2\ d\mu\right)^2 + 2 \left(\int_\R x\ d\mu\right)^3\\
&\dots
\end{align*}
are the \emph{free cumulants} of $\mu$.  

\begin{example}\label{ex:semicircle} If $\semicirc$ is the \emph{semicircular distribution}
$$ \semicirc \coloneqq \frac{1}{2\pi} (4-x^2)_+^{1/2}\ dx$$
then one easily verifies that
$$ R_\semicirc(s) = s,$$
thus $\kappa_1(\semicirc)=1$ and $\kappa_n(\semicirc)=0$ for $n>1$.  It is not difficult to see that a compactly supported probability measure $\mu$ is uniquely determined by its $R$-transform $R_\mu$.
\end{example}
 
The free convolution $\mu \boxplus \nu$ is then the unique compactly supported measure for which
$$ R_{\mu \boxplus \nu}(s) = R_\mu(s) + R_\nu(s)$$
for all sufficiently small $s$, or equivalently $\kappa_n(\mu\boxplus \nu) = \kappa_n(\mu) + \kappa_n(\nu)$ for all $n \geq 1$ (see, e.g., \cite{vdn}); this is a commutative and associative operation on such measures.  If $k$ is a positive integer, one can then $\mu^{\boxplus k} = \mu \boxplus \dots \boxplus \mu$ to be the free convolution of $k$ copies of $\mu$, and one clearly has
\begin{equation}\label{rbox}
 R_{\mu^{\boxplus k}}(z) = k R_\mu(z)
\end{equation}
for all sufficiently small $s$, or equivalently 
\begin{equation}\label{kappan}
\kappa_n(\mu^{\boxplus k}) = k \kappa_n(\mu)
\end{equation}
for all $n \geq 1$.  One can normalize these free convolutions by defining the dilates $\lambda_* \mu$ of a probability measure $\mu$ by a scaling factor $\lambda>0$ to be the pushforward of $\mu$ by the dilation $x \mapsto \lambda x$ (thus, if $\mu$ is the law of a random variable $X$, then $\lambda_* \mu$ is the law of $\lambda X$).  One easily verifies the scaling laws
\begin{equation}\label{gmu}
 G_{\lambda_* \mu}(z) = \lambda^{-1} G_\mu( z / \lambda )
\end{equation}
for all $z$ outside of the support of $\mu$, and
$$ R_{\lambda_* \mu}(s) = \lambda R_\mu(\lambda s)$$
for all sufficiently small $s$ (or equivalently $\kappa_n(\lambda_* \mu) = \lambda^n \kappa_n(\mu)$ for all $n \geq 1$), hence one has
$$ R_{k^{-1/2}_* \mu^{\boxplus k}}(s) = k^{1/2} R_\mu(k^{-1/2} s).$$
Using this relation, Voiculescu \cite{voi} established the \emph{free central limit theorem}: if $\mu$ is a compactly supported probability measure of mean zero and variance one, then the normalized free convolutions $k^{-1/2}_* \mu^{\boxplus k}$ converge in the vague topology to the semicircular distribution $\semicirc$.

In \cite{voi-entropy}, Voiculescu also introduced the \emph{free entropy}	
$$ \chi(\mu) \coloneqq \int_\R \int_\R \log|x-y|\ d\mu(x)d\mu(y) + \frac{3}{4} + \frac{1}{2} \log 2\pi$$
and the \emph{free Fisher information}\footnote{There appears to be some inconsistency in terms of normalization constants in the definition of $\Phi$ between (and within) Voiculescu's papers \cite{voi-entropy,voi-score}. In particular, there appears to be an unfortunate typo in the statement and proof of Lemma 3.2 of \cite{voi-entropy}, in which a factor of $\pi^2/2$ was left off.  Our choice of normalization in the definition of $\Phi$ is compatible with its definition via the $L^2$ norm of a free conjugate variable as in \cite{voi-score} and differs by a factor of $4\pi^2/3$ from the definition in \cite{voi-entropy}.  If $\mu$ is the semicircular law with second moment equal to $1$ as in Example~\ref{ex:semicircle}, then its free Fisher information equals $1$ in our normalization.}
\begin{equation}\label{23-phi}
 \Phi(\mu) \coloneqq \frac{4\pi^2}{3} \int_\R \left(\frac{d\mu}{dx}\right)^3\ dx
\end{equation}
for compactly supported probability measures $\mu$ (with the convention that $\Phi(\mu)=+\infty$ if $\mu$ is not absolutely continuous); the two concepts are related by the derivative rationa
$$ \Phi(\mu) = 2 \frac{d}{dt} \chi( \mu \boxplus \sqrt{t}_* \semicirc )|_{t=0}$$
and the closely associated integral formula
\begin{equation}\label{tik}
 \chi(\mu) = \frac{1}{2} \int_0^\infty \left(\frac{1}{1+t} - \Phi(\mu \boxplus \sqrt{t}_* \semicirc)\right)\ dt + \frac{1}{2} \log 2\pi e.
\end{equation}
In \cite{shly}, it was shown that these quantities were monotone with respect to normalized free convolution powers in the sense that
\begin{equation}\label{ch-1}
 \chi( (k+1)^{-1/2}_* \mu^{\boxplus k+1} ) \geq \chi( k^{-1/2}_* \mu^{\boxplus k} ) 
\end{equation}
and
\begin{equation}\label{ch-2}
 \Phi( (k+1)^{-1/2}_* \mu^{\boxplus k+1} ) \leq \Phi( k^{-1/2}_* \mu^{\boxplus k} ) 
\end{equation}
for all compactly supported $\mu$, and all $k \geq 1$.  This was the free analog of a corresponding result proven in \cite{abbn} for the Shannon entropy and classical Fisher information, answering a question of Shannon \cite{shannon}. 

As is customary, if $X$ is a real noncommutative random variable with law $\mu$, we write $G_X \coloneqq G_\mu$, $R_X \coloneqq R_\mu$, $\kappa_n(X) \coloneqq \kappa_n(\mu)$, $\Phi(X) \coloneqq \Phi(\mu)$, and $\chi(X) \coloneqq \chi(\mu)$.

\subsection{Fractional free convolution powers}

Observe that the right-hand sides of \eqref{rbox}, \eqref{kappan} make sense for any real number $k$.  This raises the question of whether one can define fractional powers $\mu^{\boxplus k}$ for non-integer choices of $k$.  This is indeed true:

\begin{proposition}[Existence of fractional free convolution powers]\label{exist}  Let $\mu$ be a compactly supported probability measure on $\R$, and let $k \geq 1$ be real.  Then there exists a unique compactly supported probability measure $\mu^{\boxplus k}$ on $\R$ such that
\begin{equation}\label{rmos}
 R_{\mu^{\boxplus k}}(s) = k R_\mu(s)
\end{equation}
for all sufficiently small $s$, or equivalently
$$ \kappa_n(\mu^{\boxplus k}) = k^n \kappa_n(\mu)$$
for all $n \geq 1$.  
\end{proposition}

Thus for instance $\semicirc^{\boxplus k} = k^{1/2}_* \semicirc$ for any $k \geq 1$.

Proposition \ref{exist} was first established for sufficiently large $k$ by Bercovici and Voiculescu \cite{bv}, and then for all $k \geq 1$ by Nica and Speicher \cite{ns}; a complex analysis proof using subordination was given by Belinschi-Bercovici \cite{bb-1, bb-2} and Huang \cite{huang}. See also the recent paper \cite{belin} for further study of the subordination functions associated to these measures, and \cite{huang}, \cite{williams} for further regularity and support properties of the $\mu^{\boxplus k}$, and \cite{abfn}, \cite{shly-2} for an extension to the case when $k$ is a completely positive map and $\mu$ takes values in a $C^*$-algebra. 

From \eqref{rmos} and the invertibility of the $R$-transform we have the semigroup law
\begin{equation}\label{mono}
 (\mu^{\boxplus k})^{\boxplus l} = \mu^{\boxplus kl}
\end{equation}
for any real $k,l \geq 1$, and similarly
$$
\mu^{\boxplus k} \boxplus \mu^{\boxplus l} = \mu^{\boxplus k+l}.
$$
Thus one can now view $k \mapsto \mu^{\boxplus k}$ as a continuous one-parameter semigroup.  There are also connections between fractional free convolution powers and free \emph{multiplicative} convolution: see \cite{bn} and below. 

The proof of Proposition \ref{exist} by Nica and Speicher \cite{ns} also gave the following free probability interpretation of such powers.  Let $(\A, \tau)$ be a noncommutative probability space (that is to say, a complex associative unital $*$-algebra $\A$ equipped with a unital tracial positive linear functional $\tau$, and let $p \in \A$ be a self-adjoint projection of trace $1/k$ for some $k \geq 1$ (thus $p^* = p^2 = p$ and $\tau(p)=1/k$).  Then we can form another noncommutative probability space $(\A_p, \tau_p)$ by defining $\A_p = [p\A p]$ to be a copy\footnote{Thus for instance $[pXp][pYp] = [pXppYp] = [p(XpY)p]$ and $[pXp]+[pYp]=[pXp+pYp] = [p(X+Y)p]$.  The brackets $[]$ are a formal symbol, which we introduce in order to distinguish the algebraic structures of $\A_p$ from that of $\A$.  In particular, the unit $1 = [p]$ of $\A_p$ needs to be distinguished from the non-unit $p$ of $\A$, and the invertibility of an element $[pXp]$ of $\A_p$ does not imply the invertibility of the corresponding element $pXp$ of $\A$.} 
$$
\A_p \coloneqq \{ [pXp]: X \in \A\}$$
of $p\A p \coloneqq \{ pXp: X \in \A \}$, and
\begin{equation}\label{taup}
\tau_p( [pXp] ) \coloneqq k \tau(pXp) = k \tau(pX) = k \tau(Xp)
\end{equation}
for any $X \in \A$.  It is not difficult to verify that $(\A_p, \tau_p)$ is a noncommutative probability space.  We have a ``minor map'' or ``compression map'' $\pi \colon \A \to \A_p$ defined by
$$ \pi(X) \coloneqq [pXp];$$
this map is $*$-linear, surjective, and maps the unit $1$ of $\A$ to the unit $1 = [p]$ of $\A_p$.  The minor map $\pi$ is not an algebra homomorphism nor is it trace-preserving, but one does at least have homomorphism-like identities
\begin{equation}\label{piy}
\pi(X) \pi(Y) = \pi(pXpYp) = \pi( XpYp ) = \pi(pXpY) = \pi(XpY) 
\end{equation}
for any $X,Y \in \A$, and from \eqref{taup} we have
\begin{equation}\label{taup-2}
\tau_p( \pi(X) ) \coloneqq k \tau(pXp)
\end{equation}
for any $X \in \A$.

\begin{example} Let $k$ be a rational number $k = N/M > 1$, $\A = M_N(\C)$ be the space of $N \times N$ matrices with trace $\tau(X) \coloneqq \frac{1}{N} \Tr(X)$, and $p = \begin{pmatrix} I_M & 0_{M \times N-M} \\ 0_{N-M \times M} & 0_{N-M \times N-M}\end{pmatrix}$ be the orthogonal projection to span of the first $M$ standard basis vectors.  Then $\A_p$ can identified with $M_M(\C)$ (with trace $\tau_p(X) \coloneqq \frac{1}{M} \Tr(X)$).  With this identification, $\pi(X)$ is the upper left $M \times M$ minor of $X$.
\end{example}

We then have the following interpretation of fractional free convolution powers as a normalized free minor process.

\begin{proposition}[Fractional free convolution powers from free minors]\label{minor-free}  If $(\A,\tau)$ is a noncommutative probability space, $k \geq 1$ is real, $p$ is a real projection of trace $1/k$, and $X \in \A$ has some law $\mu$ and is freely independent of $p$, then $k \pi(X)$ has law $\mu^{\boxplus k}$.  Thus
$$ R_{k \pi(X)}(s) = k R_X(s)$$
or equivalently
\begin{equation}\label{pipx}
 R_{\pi(X)}(s) = R_X(s/k)
\end{equation}
for all sufficiently small $s$; in terms of free cumulants, this becomes
\begin{equation}\label{kip}
 \kappa_n( \pi(X) ) = k^{1-n} \kappa_n(X)
\end{equation}
for $n \geq 1$.
\end{proposition}

\begin{proof} See \cite[Corollary 1.14]{ns}.  For the convenience of the reader, we also give a self-contained proof in Appendix \ref{power}.
\end{proof}

\begin{remark}
By the asymptotic free independence of independent unitarily invariant large matrices (see appendix to \cite{ns}), one can also define $\mu^{\boxplus k}$ for any real $k \geq 1$ as the asymptotic empirical distribution of the $M \times M$ random matrix $k A_{M \times M}$ as $N \to \infty$, where $A$ is a $N \times N$ bounded random Hermitian matrix, invariant under unitary conjugation, whose empirical law converges to $\mu$, $M \coloneqq \lceil N/k \rceil$, and $A_{M \times M}$ is the upper left $M \times M$ minor of $A$.  There is a similar interpretation of fractional free convolution powers in terms of the asymptotic distribution of large random Young tableaux, drawn uniformly from all tableaux of a given shape; see \cite{biane}.
\end{remark}

One can investigate the dynamic of fractional free convolution powers as follows. From \eqref{rmos}, \eqref{rg} one has
\begin{equation}\label{eqo}
 \frac{1}{G_{\mu^{\boxplus k}(z)}} + k R_\mu( G_{\mu^{\boxplus k}}(z) ) = z
\end{equation}
for all $k \geq 1$ ranging in a compact set and all sufficiently large $z$.  In particular, from the inverse function theorem, $G_{\mu^{\boxplus k}}(z)$ varies smoothly in $k,z$ in this regime.  Applying the first order differential operator
$$ \partial_z G_{\mu^{\boxplus k}}(z) \partial_k -  \partial_k G_{\mu^{\boxplus k}}(z) \partial_z,$$
which annihilates $G_{\mu^{\boxplus k}}(z)$ as well as any autonomous function of $G_{\mu^{\boxplus k}(z)}$, to both sides of \eqref{eqo}, we conclude that
$$ \left(\partial_z G_{\mu^{\boxplus k}}(z)\right) R_\mu( G_{\mu^{\boxplus k}}(z) ) = - \partial_k G_{\mu^{\boxplus k}}(z)$$
which when combined with \eqref{eqo} to eliminate the $R_\mu( G_{\mu^{\boxplus k}}(z) )$ factor yields the Burgers-type equation
\begin{equation}\label{diffeq}
(k \partial_k + z \partial_z) G_{\mu^{\boxplus k}}(z) = \frac{\partial_z G_{\mu^{\boxplus k}}(z)}{G_{\mu^{\boxplus k}}(z)}  
\end{equation}
for $k \geq 1$ in a fixed compact region and sufficiently large $z$.  From \eqref{gmu} we have
$$ G_{k^{-1/2}_* \mu^{\boxplus k}}(z) = k^{1/2} G_{\mu^{\boxplus k}}(k^{1/2} z)$$
so after some calculation we can also write this equation in renormalized form as
\begin{equation}\label{diffeq-2}
(k \partial_k + \frac{1}{2} z \partial_z) G_{k^{-1/2}_* \mu^{\boxplus k}}(z) = \frac{ \partial_z G_{k^{-1/2}_* \mu^{\boxplus k}}(z)}{G_{k^{-1/2}_* \mu^{\boxplus k}}(z)}  + \frac{1}{2} G_{k^{-1/2}_* \mu^{\boxplus k}}(z).
\end{equation}
This in turn gives a differential equation for $k^{-1/2}_* \mu^{\boxplus k}$; see \eqref{feq}.

It is now natural to ask whether the properties of integer free convolution powers $\mu^{\boxplus k}, k \in \N$ extend to the fractional counterparts $\mu^{\boxplus k}, k \in \R$.  For instance, fractional convolution power allow us to make sense of the law of central limit sums $Y_k := k^{-1/2} \sum_{j=1}^k X_j$ of free iid copies $X_j$ of a centered bounded random variable $X$.  If $X$ has law $\mu$, then $Y_N$ has law $k^{-1/2}_* \mu^{\boxplus k}$.  The free central limit theorem states that the law of $Y_k$ converges to the semicircle law as $k\to \infty$ along positive integers.  It is easy to see that the $R$-transform proof of the free central limit theorem (see, e.g., \cite{vdn}) shows also that $k_*^{-1/2} \mu^{\boxplus k}$ converges to the semicircle law as $k \to \infty$ along the positive reals.
 
Now we turn to the monotonicity of free entropy and free Fisher information, which is the first main result of our paper.

\begin{theorem}[Monotonicity of free entropy and free Fisher information]\label{main}  Let $\mu$ be a compactly supported finite probability measure.  Then $\chi( k^{-1/2}_* \mu^{\boxplus k} )$ is monotone non-decreasing and $\Phi(k^{-1/2}_* \mu^{\boxplus k} )$ is monotone non-increasing in $k$ for real $k \geq 1$.
\end{theorem}

Specializing to the case of integer $k$, we recover the previous results \eqref{ch-1}, \eqref{ch-2}.

We prove this theorem in Section \ref{score-proof}.  Our argument relies on the characterization of fractional free convolution powers in Proposition \ref{minor-free}, together with the fundamental fact that free independence is preserved by taking (free) minors. 
This proof also allows for an extension to several variables; see Theorem \ref{sev-var}.  In fact, as was shown to us by David Jekel, by applying a similar argument to the classical entropy and Fisher information of random matrix models, the argument can be adapted to a microstate setting, allowing one to also prove monotonicity for Voiculescu's multivariable microstates free entropy introduced in \cite{voi-entropy2}; see Appendix \ref{sec-microstates}.   Our argument shows that equality in Theorem \ref{main} only holds when $\mu$ is a rescaled version of semicircular measure $\semicirc$; see Proposition \ref{equality}.

By  computing all of the quantities that appear explicitly or implicitly in the proof given in Section \ref{score-proof}, we were able to extract a complex analytic proof of Theorem \ref{main} using the differential equation \eqref{diffeq}, at least if one assumes additional regularity on the original measure $\mu$; we present a streamlined (but somewhat unmotivated) version of this proof in Section \ref{complex-sec}. 

The fact that the flow \eqref{diffeq} enjoys some monotonicity properties suggests that it has an interpretation as a gradient flow.  We were not able to obtain such an interpretation, but we instead were able to find a (formal) \emph{Lagrangian} interpretation of this flow, when viewed in ``Gelfand-Tsetlin coordinates''.  Namely, let $\mu$ be a compactly supported probability measure on $\R$, let $\Delta$ denote the ``Gelfand-Tsetlin pyramid''
$$ \Delta \coloneqq \{ (s,y): 0 < s < 1; 0 < y < s \},$$
and for any $(s,y) \in \R$ let $\lambda(s,y)$ denote the real number for which
\begin{equation}\label{mus}
 \mu^{\boxplus 1/s}( (-\infty, \lambda(s,y)/s] ) = y/s.
\end{equation}
Under suitable non-degeneracy assumptions on $\mu$, $\lambda(s,y)$ will be well-defined and vary smoothly with $s,y$.  This function $\lambda(s,y)$ has the following random matrix interpretation.  Let $N$ be a large natural number parameter, and let $A$ be a random Hermitian $N \times N$ matrix, invariant under unitary conjugation, and with empirical spectral distribution converging to $\mu$ as $N \to \infty$.  Then the $\lceil yN\rceil^{\mathrm{th}}$ smallest eigenvalue of the $\lceil sN \rceil \times \lceil sN \rceil$ minor will be concentrated around $\lambda(s,y)$.  In Section \ref{variation-sec} we establish

\begin{theorem}[Variational formulation]\label{var-thm}  Formally, $\lambda$ is a critical point of the Lagrangian
\begin{equation}\label{lag}
 \int_\Delta L( \partial_s \lambda, \partial_y \lambda)\ ds dy
\end{equation}
where the Lagrangian density $L$ is given by the formula
\begin{equation}\label{L-def}
 L(\lambda_s, \lambda_y) \coloneqq \log \lambda_y + \log \sin \pi \frac{\lambda_s}{\lambda_y}.
\end{equation}
\end{theorem}

We do not have a satisfactory interpretation of this Lagrangian density $L$.  In \cite{met} it is shown that random Gelfand-Tsetlin patterns formed by taking eigenvalues of successive minors asymptotically have the law of the Boutillier bead process \cite{bou}, so it seems reasonable to conjecture\footnote{Note added in proof: the recent calculations of local entropy (or ``surface-tension'') of the bead process in \cite{sun} (see also \cite{john}) seem to strongly support this conjecture.  We thank Istvan Prause for these references.  Furthermore, it was pointed out to us by Vadim Gorin (private communication) that the random Gelfand-Tsetlin process is a continuous version of a random lozenge tiling \cite{gorin}, for which a variational description was provided in \cite{cohn}, and that the calculation in \cite{sun} can be viewed as a careful evaluation of the continuum limit of the theory in \cite{cohn}.  A very similar conjecture in the context of random Young tableaux has recently been proposed in \cite{gordenko}.} that the Lagrangian density $L(\lambda_s, \lambda_y)$ is proportional to the entropy of this process (with density proportional to $1/\lambda_y$, and drift velocity proportional to $\lambda_s/\lambda_y$).

\subsection{Acknowledgments}

The first author was partially supported by NSF grant DMS-1762360. The second author was partially supported by NSF grant DMS-1764034 and by a Simons Investigator Award.  This project was initiated during the IPAM program for Quantitative Linear Algebra in 2018. We thank Vadim Gorin, Istvan Prause and Stefan Steinerberger for providing recent relevant references, and David Jekel for providing Appendix \ref{sec-microstates}.  David Jekel was supported by NSF grant DMS-2002826. Finally, we thank the anonymous referee for careful reading of the manuscript and several useful suggestions and corrections.

\section{Proof of monotonicity}\label{score-proof}

We now prove Theorem \ref{main}.  We will rely on two main tools.  The first is the fact that free independence is preserved by taking free minors:

\begin{lemma}\label{indep-minor}  Let $(\A, \tau)$ be a noncommutative probability space, and let $p \in \A$ be a real projection.  If $B_1,\dots,B_n \in \A$ are unital algebras such that $B_1,\dots,B_n,p$ are free in $\A$, then $\pi(B_1),\dots,\pi(B_n)$ are free in $\A_p$. 
\end{lemma}

\begin{proof}  See \cite[Corollary 1.12]{ns}.
\end{proof}

Next we recall the notion of \emph{free score} (also called free conjugate variable) from \cite{voi-score}.  If $(\A, \tau)$ is a noncommutative probability space, $X \in \A$, and $B$ is a unital subalgebra of $\A$, we define the \emph{free score} $J(X:B)$ of $X$ relative to $B$ (if it exists) to be the unique element in the $L^2(\tau)$ closure of the algebra $\mathrm{Alg}(X,B)$ generated by $X$ and $B$ with the property that
\begin{equation}\label{tautau}
 \frac{d}{d\eps} \tau( Z P( X + \eps Z, Y_1,\dots,Y_n) )|_{\eps=0} = \tau( J(X:B) P(X, Y_1,\dots,Y_n) )
\end{equation}
for any $Y_1,\dots,Y_n \in B$ and any noncommutative polynomial $P(X,Y_1,\dots,Y_n)$ in $n+1$ variables, where $Z$ is a noncommutative random variable of mean zero and variance one that is freely independent of $X,B$ (such a variable always exists if one is willing to extend the noncommutative space $(\A,\tau)$.)  An equivalent definition (see \cite[Proposition 3.4]{voi-score}) is that
\begin{equation}\label{colon}
 \tau \otimes \tau( \partial P( X, Y_1,\dots, Y_n ) ) = \tau( J(X:B) P(X, Y_1,\dots,Y_n) )
\end{equation}
where $\partial \colon \mathrm{Alg}(X,B) \to L^2(\tau \otimes \tau)$ is the unique derivation such that $\partial X = 1 \otimes 1$ and $\partial Y = 0$ for all $Y \in B$, see \cite{voi-score}. If $B$ is the trivial algebra $\C$, we abbreviate $J(X:\C)$ as $J(X)$.  It is known that the free Fisher information $\Phi(X)$ is finite if and only if the score exists, in which case \cite{voi-score} 
\begin{equation}\label{phix}
 \Phi(X) = \| J(X) \|_{L^2(\tau)}^2 = \tau( J(X)^2 );
\end{equation}
indeed this can be viewed as the ``true'' definition of the free Fisher information.  Specializing \eqref{tautau} to the case $P=1$ we see that the score, if it exists, is always trace-free:
\begin{equation}\label{score-trace}
\tau( J(X:B) ) = 0.
\end{equation}

We have the following basic fact from \cite{voi-score}:

\begin{lemma}[Free extensions do not affect free score]\label{freescore}  Let $(\A,\tau)$ be a noncommutative probability space, let $B, B'$ be unital subalgebras of $\A$, and $X \in \A$ be such that $X,B$ are free from $B'$. The score $J(X:B)$ exists if and only if the score $J(X:B,B')$ exists, and the two scores are equal: $J(X:B) = J(X:B,B')$.  Here we use $B,B'$ to denote the algebra generated by $B$ and $B'$.
\end{lemma}

\begin{proof}
See \cite[Proposition 3.6]{voi-score}.
%
\end{proof}

Now we come to a basic identity.

\begin{proposition}[Free score and minors]\label{score-minor}  Let $(\A,\tau)$ be a noncommutative probability space, let $p \in \A$ be a real projection of trace $k^{-1}$ for some $k \geq 1$, let $X \in \A$, and let $B$ be a unital subalgebra of $\A$.  Assume that $X,B$ are free of $p$ and that the free score $J(X:B)$ exists.  Then the free score $J(\pi(X):\pi(B))$ exists and is equal to 
$$ J(\pi(X):\pi(B)) = k \E\left( \pi(J(X:B)) | \pi(X), \pi(B) \right)$$
where $\E(\cdot  | \pi(X), \pi(B) )$ denotes the orthogonal projection (or conditional expectation) in $L^2(\tau_p)$ to the subalgebra of $\A_p$ generated by $\pi(X)$ and $\pi(B)$.
\end{proposition}

\begin{proof}  Let $Z$ be a noncommutative random variable in $\A$ of mean zero and variance $1$ that is free from $X,B$; such a variable exists after extending $\A$ if necessary.  From Lemma \ref{indep-minor}, $k^{1/2} \pi(Z) \in \A_p$ has mean zero and variance $1$, and is free from $\pi(X), \pi(B)$.   By definition of free score, it thus suffices to establish the identity
\begin{align*}
& \frac{d}{d\eps} \tau_p\left( k^{1/2} \pi(Z) P( \pi(X)+k^{1/2}\eps \pi(Z), \pi(B)) \right)|_{\eps=0}\\
&\quad  = \tau_p\left( k \E\left( \pi(J(X:B)) | \pi(X), \pi(B) \right) P(\pi(X), \pi(B) )\right)
\end{align*}
for any polynomial $P( \pi(X), \pi(B) )$.  By the chain rule we may cancel the factors of $k^{1/2}, k$, and as $P(\pi(X), \pi(B) )$ lies in the range of the orthogonal projection $\E( | \pi(X), \pi(B) )$ we may delete the projection, thus we now need to show
$$ \frac{d}{d\eps} \tau_p\left( \pi(Z) P( \pi(X)+\eps \pi(Z), \pi(B)) \right)|_{\eps=0} = \tau_p\left( \pi(J(X:B)) P(\pi(X), \pi(B) )\right).$$
Using the definition of $\tau_p$ and $\pi$ and the idempotent nature of $p$ this is equivalent to
$$ \frac{d}{d\eps} \tau\left( \pi(Z) P( p (X + \eps Z) p, p B p) \right)|_{\eps=0} = \tau\left( J(X:B) P(p X p, p B p )\right).$$
By Lemma \ref{freescore}, $J(X:B,p)$ exists and is equal to $J(X:B)$.  Applying the definition of free score to the polynomial $P( p (X + \eps Z) p, p B p)$, we obtain the claim.
\end{proof}

Specializing this proposition to the case when $B = \C$, we conclude that (if the free score $J(X)$ exists)
$$ J(\pi(X)) = k \E( \pi(J(X)) | \pi(X) )$$
and hence by Pythagoras' theorem
$$ \Phi( \pi(X) ) = \| J(\pi(X)) \|_{L^2(\tau_p)}^2 \leq k^2 \| \pi(J(X))\|_{L^2(\tau_p)}^2.$$
As $J(X)$ lies in the closure of the algebra generated by $X$, it is free of $p$, thus by \eqref{taup-2} and free independence
$$ \| \pi(J(X))\|_{L^2(\tau_p)}^2 = k \tau( p J(X) p J(X) p ) = k^{-1} \tau(J(X)^2).$$
We conclude the inequality
\begin{equation}\label{phip}
\Phi( \pi(X) ) \leq k \Phi(X).
\end{equation}
Using the easily verified scaling
\begin{equation}\label{easy}
 \Phi(\lambda X) = \lambda^{-2} \Phi(X)
\end{equation}
for any $\lambda>0$, we conclude that
$$ \Phi( k^{1/2} \pi(X) ) \leq \Phi(X)$$
whenever $J(X)$ exists.  Clearly this inequality also holds when $J(X)$ does not exist, since the right-hand side is infinite.  We thus have
\begin{equation}\label{ko}
\Phi( k^{-1/2}_* \mu^{\boxplus k} ) \leq \Phi(\mu)
\end{equation}
for any $k \geq 1$ and any compactly supported $\mu$. Rescaling using \eqref{mono}, \eqref{easy} we obtain the non-increasing nature of $\Phi(\chi( k^{-1/2}_* \mu^{\boxplus k} )$.  To obtain the corresponding monotonicity for free entropy, we use \eqref{tik}, \eqref{easy} to compute
\begin{align*}
 \chi(k^{-1/2}_* \mu^{\boxplus k}) &= \frac{1}{2} \int_0^\infty \left(\frac{1}{1+t} - \Phi(k^{-1/2}_* \mu^{\boxplus k} \boxplus \sqrt{t}_* \semicirc)\right)\ dt + \frac{1}{2} \log 2\pi e \\
&= \frac{1}{2} \int_0^\infty \left(\frac{1}{1+t} - \Phi(k^{-1/2}_* (\mu \boxplus \sqrt{t}_* \semicirc)^{\boxplus k})\right)\ dt + \frac{1}{2} \log 2\pi e 
\end{align*}
and the non-increasing nature of $\chi(k^{-1/2}_* \mu^{\boxplus k})$ then follows from the non-increasing nature of $\Phi(k^{-1/2}_* (\mu \boxplus \sqrt{t}_* \semicirc)^{\boxplus k}))$ for each $t \geq 0$.

The above argument generalizes to also obtain analogous monotonicity properties for the (non-microstate) free entropy and free Fisher information of several variables.  We recall from \cite{voi-score} that the \emph{relative free Fisher information} $\Phi^*(X:B)$ of a noncommutative real random variable $X \in \A$ relative to an algebra $B$ is given by the formula
$$ \Phi^*(X:B) = \| J(X:B) \|_{L^2(\tau)}^2 = \tau( J(X:B)^2 ),$$
and the non-microstate free Fisher information $\Phi^*(X_1,\dots,X_n)$ of a finite number of noncommutative real random variables $X_1,\dots,X_n \in \A$ is given by the formula
\begin{equation}\label{ai}
 \Phi^*(X_1,\dots,X_n) \coloneqq \sum_{i=1}^n \Phi^*(X_i : X_1,\dots,X_{i-1},X_{i+1},\dots,X_n).
\end{equation}
The corresponding non-microstate free entropy $\chi^*(X_1,\dots,X_n)$ is then defined as
$$ 
 \chi^*(X_1,\dots,X_n) = \frac{1}{2} \int_0^\infty\left(\frac{n}{1+t} - \Phi^*(X_1 + t^{1/2} Z_1, \dots, X_n + t^{1/2} Z_n)\right)\ dt + \frac{n}{2} \log 2\pi e$$
where $Z_1,\dots,Z_n$ are semicircular elements that are free from each other and from $X_1,\dots,X_n$. 

\begin{theorem}[Monotonicity for several variables]\label{sev-var} If $X_1,\dots,X_n \in \A$, $k \geq 1$, and $p$ is a real projection of trace $1/k$ that is free from $X_1,\dots,X_n$, one has
$$ \Phi^*( k^{1/2} \pi(X_1), \dots, k^{1/2} \pi(X_n) ) \leq \Phi^*( X_1,\dots,X_n)$$
and
$$ \chi^*(k^{1/2} \pi(X_1),\dots, k^{1/2} \pi(X_n) ) \geq \chi^*( X_1,\dots,X_n).$$
\end{theorem}

We remark that an easy rescaling gives the equivalent forms
$$ \Phi^*( \pi(X_1), \dots, \pi(X_n) ) \leq k \Phi^*( X_1,\dots,X_n)$$
and
$$ \chi^*(\pi(X_1),\dots, \pi(X_n) ) \geq \chi^*( X_1,\dots,X_n) -  \frac{n}{2} \log k.$$
of these inequalities.

\begin{proof} It suffices to prove the former inequality, as the latter follows by repeating the previous arguments.  From \eqref{ai} it suffices to show that
\begin{gather*} \Phi^*(k^{1/2} \pi(X_i) : \pi(X_1),\dots,\pi(X_{i-1}),\pi(X_{i+1}),\dots,\pi(X_n)) \\ 
\leq \Phi^*(X_i : X_1,\dots,X_{i-1},X_{i+1},\dots,X_n)\end{gather*}
for each $i=1,\dots,n$.  Let $B$ be the algebra generated by $X_1,\dots,X_{i-1},X_{i+1},\dots,X_n$, then we can rewrite this inequality as
$$ k^{-1} \| J( \pi(X_i) : \pi(X_1),\dots,\pi(X_{i-1}),\pi(X_{i+1}),\dots,\pi(X_n)) \|_{L^2(\tau_p)}^2
\leq \| J(X_i:B) \|_{L^2(\tau)}^2.$$
From Proposition \ref{score-minor} and Pythagoras' theorem we see that if $J(X_i:B)$ exists, then so does $J(\pi(X_i):\pi(B))$ and
$$ \| J(\pi(X_i):\pi(B)) \|_{L^2(\tau_p)}^2 \leq k^2 \| \pi(J(X:B)) \|_{L^2(\tau_p)}^2
= k \| J(X:B) \|_{L^2(\tau)}^2$$
where we as before we use the fact that $J(X:B)$ is in the closure of the algebra generated by $X_1,\dots,X_n$ and is hence free of $p$.

The algebra $B'$ generated by $\pi(X_1),\dots,\pi(X_{i-1}),\pi(X_{i+1}),\dots,\pi(X_n)$ is a subalgebra of $\pi(B)$, hence the score $J(\pi(X_i):B')$ exists and is a projection of $J(\pi(X_i):\pi(B))$.  By a further application of Pythagoras, we conclude that
$$
\| J( \pi(X_i) : \pi(X_1),\dots,\pi(X_{i-1}),\pi(X_{i+1}),\dots,\pi(X_n)) \|_{L^2(\tau_p)}^2 \leq
k \| J(X:B) \|_{L^2(\tau)}^2$$
and the claim follows.
\end{proof}

\begin{remark}
Appendix \ref{sec-microstates} establishes monotonicity of entropy for $n$-tuples for the so-called microstates free entropy $\chi$, introduced by Voiculescu in \cite{voi-entropy2}.
\end{remark}

Returning to the case of a single variable, we can analyze the above proof of monotonicity further to extract when equality occurs:

\begin{proposition}[Characterization of equality]\label{equality}  Let $\mu$ be a compactly supported real probability measure with $\Phi(\mu) < \infty$, and let $k > 1$.  If $\Phi(k^{-1/2}_* \mu^{\boxplus k}) = \Phi(\mu)$, then $\mu$ is the law of $\alpha + \beta u$ for some semicircular element $u$, real $\alpha$, and $\beta>0$.
\end{proposition}

Conversely, it is easy to see that if $\mu$ is the law of $\alpha+\beta u$ for a semicircular $u$, then $k^{-1/2}_* \mu^{\boxplus k}$ is the law of $k^{1/2} \alpha + \beta u$, so that $\Phi(k^{-1/2}_* \mu^{\boxplus k}) = \Phi(\mu)$.  Using the representation \eqref{tik} we see that we also have an analogous claim with the free Fisher information $\Phi$ replaced by the free entropy $\chi$.

\begin{proof}  By translating $\mu$ (which does not affect the free Fisher information of $\mu$ or $k^{-1/2}_* \mu^{\boxplus k}$) we may assume that $\mu$ has mean zero.  We can also assume that $\mu$ is not a point mass as the free Fisher information is infinite in that case.  Inspecting the proof of \eqref{ko}, we must have
$$ \tau_p( (k \E( \pi(J(X)) | \pi(X) ))^2 ) \leq k^2 \tau_p( \pi(J(X))^2 )$$
and thus $\pi(J(X))$ lies in the $L^2$ closure of the algebra generated by $\pi(X)$.  In particular, these two variables commute, so that
$$ \tau_p( \pi(J(X)) \pi(J(X)) \pi(X) \pi(X) ) = 
\tau_p( \pi(J(X)) \pi(X) \pi(J(X)) \pi(X) ) 
$$
(note that both sides are finite by Cauchy-Schwarz); by \eqref{taup} we thus have
\begin{equation}\label{pjx}
 \tau( p J(X) p J(X) p X p X ) = \tau( p J(X) p X p J(X) p X ).
\end{equation}
The variables $X, J(X)$ are free of $p$ (since $J(X)$ lies in the closure of the algebra generated by $X$), and have trace zero by hypothesis and \eqref{score-trace}.  Splitting $p$ into the trace $1/k$ and the trace-free part $p' \coloneqq p - \frac{1}{k}$, we obtain $2^4$ terms, but from free independence the only terms that survive are those that involve either zero or two copies of $p'$, and in the latter case the $p'$ terms need to be separated from each other cyclically by two of the $X,J(X)$ factors.  In other words, we have
\begin{align*} \tau( p J(X) p J(X) p X p X ) &= k^{-4} \tau(J(X)^2 X^2) + k^{-2} \tau(p' J(X)^2 p' X^2 ) \\
&\quad + k^{-2} \tau( J(X) p' J(X) X p' X )
\end{align*}
and similarly
\begin{align*} \tau( p J(X) p X p J(X) p X ) &= k^{-4} \tau(J(X) X J(X) X) + k^{-2} \tau(p' J(X) X p' J(X) X ) \\
&\quad + k^{-2} \tau( J(X) p' X J(X) p' X ).
\end{align*}
Applying these identities to \eqref{pjx} and noting that $J(X)$ commutes with $X$, we conclude that
$$ \tau(p' J(X)^2 p' X^2 ) = \tau(p' J(X) X p' J(X) X ).$$
From free independence we see that
$$ \tau(p' J(X)^2 p' X^2 ) = \tau((p')^2) \tau(J(X)^2) \tau(X^2)$$
and similarly
$$ \tau(p' J(X)X p' J(X)X ) = \tau((p')^2) \tau(J(X)X) \tau(J(X)X).$$
Thus we have
$$\tau(J(X)^2) \tau(X^2) = \tau(J(X)X) \tau(J(X)X)$$
which by the converse to Cauchy-Schwarz applied to the $L^2(\tau)$ inner product implies that $J(X)$ is a scalar multiple of $X$.  To finish the proof, we can either invoke the equality case of the free Stam inequality in \cite{voi-score}, or argue as follows.
The identity $J(X)=\alpha X$ for a scalar $\alpha$ implies that, if $\mu$ is the law of $X$, 
\begin{align*}\int \frac{x}{z-x} d\mu &= \tau (X (z-X)^{-1})\\
&= \alpha \tau\otimes \tau (\partial (z-X)^{-1}) \\ &= \alpha \iint \frac{ \frac{1}{z-s} - \frac{1}{z-t}}{s-t} d\mu (s) d\mu (t) \\
&= \alpha \iint \frac{1}{(z-s)(z-t)} d\mu (s) d\mu(t)  \\
&= \alpha G_\mu(z)^2,
\end{align*} 
where $G_\mu (z) = \int \frac{1}{z-s}d\mu(s)$ is the Cauchy transform.  Using the identity $$\int \frac{x}{z-x}d\mu(x)=\int \frac {x-z}{z-x}d\mu(x) + \int \frac{z}{z-x}d\mu(x) = -1+zG_\mu(z)$$ we deduce that $$ -1+zG_\mu(z)=\alpha G_\mu^2(z).$$ 
Solving this quadratic equation for $G_\mu$ (recalling that $G_\mu$ maps the upper half-plane to the lower half-plane and that $\mu$ is a probability measure) shows that $\mu$ is a scalar multiple of the semicircle law.
\end{proof}

It remains an interesting open problem to obtain an analogous characterization of equality in Theorem \ref{sev-var}.

\section{Complex analytic proof}\label{complex-sec}

We now give a direct proof of Theorem \ref{main} using the differential equation \eqref{diffeq-2}.  To avoid technicalities we will work at a somewhat formal level, ignoring some questions of convergence and regularity, or justifying operations such as integration by parts, although we will still need to be careful when handling the limiting contribution of singular integrals involving kernels such as $\frac{1}{z-w}$ or $\frac{1}{(z-w)^2}$ when $z,w$ are close.  We will also assume that the measures $k^{-1/2}_* \mu^{\boxplus k}$ take the absolutely continuous form
$$ d(k^{-1/2}_* \mu^{\boxplus k}) = f_k(x)\ dx$$
for $k \geq 1$ and $x \in \R$, where $f_k$ is compactly supported in $x$ for each $k$ and is assumed to obey sufficient regularity\footnote{It is likely that these regularity hypotheses can be removed by a limiting argument to recover Theorem \ref{main} in full generality.  For instance, one can take advantage of the fact that if $d\nu_\eps = \frac{1}{\pi} \frac{\eps}{x^2+\eps^2}\ dx$ is the Cauchy distribution with parameter $\eps>0$, then $\mu * \nu_\eps = \mu \boxplus \nu_\eps$ and $(\mu * \nu_\eps)^{\boxplus k} = \mu^{\boxplus k} * \nu_{k\eps}$, so one can apply the arguments in this section to the smooth measure $\mu * \nu_\eps$ (after carefully taking into account that this measure is no longer compactly supported), and then taking limits as $\eps \to 0$.  We leave the details to the interested reader.} in $k,x$ to justify the manipulations in the sequel.  We abbreviate
\begin{equation}\label{G-form}
 G_k(z) \coloneqq G_{k^{-1/2}_* \mu^{\boxplus k}}(z) = \int_\R \frac{f_k(x)}{z-x}\ dx.
\end{equation}
As is well known, the limiting values
$$ G_k(y+i0^\pm) \coloneqq \lim_{\eps \to 0^+} G_k(y \pm i \eps)$$
for either choice of sign $\pm$ are then given (for sufficiently regular $f$) by the Plemelj formulae
\begin{equation}\label{plemelj}
 G_k(y+i0^\pm) = \pi H f_k(y) \mp \pi i f_k(y)
\end{equation}
where 
\begin{equation}\label{H-def}
 Hf(y) \coloneqq \mathrm{p.v.} \frac{1}{\pi} \int_\R \frac{f(x)}{y-x}\ dx
\end{equation}
is the Hilbert transform of $f$.  We recall some basic identities about this Hilbert transform:

\begin{lemma}[Hilbert transform identities]\label{hilb}  If $f: \R \to \R$ is compactly supported and sufficiently regular, then one has the identities
\begin{align*}
\int_\R f(y) Hf(y)\ dy & = 0 \\
\int_\R f(y) (Hf(y))^2\ dy &= \frac{1}{3} \int_\R f(y)^3\ dy \\
H( fHf) &= \frac{(Hf)^2 - f^2}{2}
\end{align*}
\end{lemma}

\begin{proof}  Setting $G(z) \coloneqq \int_\R \frac{f(x)}{z-x}\ dx$, then by contour shifting we have
$$ \int_\R G(y+i0^+)^2\ dy = 0$$
and
$$ \int_\R G(y+i0^+)^3\ dy = 0.$$
Substituting the Plemelj formula $G(y+i0^+) = \pi H f - \pi i f$, and taking imaginary parts of both identities, we obtain the first two claims. For the final claim, we square \eqref{plemelj} to conclude that
$$ G_f(y+i0^+)^2 = \pi^2 (Hf^2 - f^2) - 2 \pi i f Hf$$
and compare this function against the function
$$ G_{2\pi f Hf}(y + i0^+) = 2 \pi^2 H(fHf) - 2 \pi i f Hf,$$
to conclude that the two holomorphic functions $G_f^2$, $G_{2\pi f Hf}$ (that both vanish at infinity) have identical imaginary parts on the half-plane, and are thus completely identical, giving the claim.
\end{proof}

\begin{remark}  From the identity $\frac{y^n-x^n}{y-x} = \sum_{j=0}^{n-1} x^j y^{n-1-j}$ for $n \geq 0$, we see that
\begin{align*}
\int_\R Hf(y) y^n\ f(y) dy &= \frac{1}{2\pi} \int_\R\int_\R f(x) f(y) \frac{y^n-x^n}{y-x}\ dx dy \\
&= \frac{1}{2\pi} \sum_{j=0}^{n-1}\left(\int_\R f(x) x^j\ dx\right) \left(\int_\R f(y) y^{n-j-1}\ dy\right).
\end{align*}
Thus if $X$ is a random variable with law $d\mu = f(x)\ dx$ for a compactly supported and sufficiently regular $f$, then on comparing the above identity with \eqref{colon} we see that the free score $J(X)$ is given by the formula
$$ J(X) = 2\pi Hf(X)$$
and thus from \eqref{phix}
$$ \Phi(X) = 4\pi^2 \int_\R f(y) Hf(y)^2\ dy.$$
Lemma \ref{hilb} shows that this formula is compatible with \eqref{23-phi}.\end{remark}

We abbreviate $f=f_1$ and $G = G_1$, and introduce the biholomorphic kernel $K(z,w)$ for $z,w \in \C \backslash \R$ by the formula
\begin{equation}\label{K-def}
 K(z,w) \coloneqq \frac{1}{G(z) G(w)} \left( \frac{G(z)-G(w)}{z-w} + G(z) G(w) \right)^2,
\end{equation}
noting that there is a removable singularity on the diagonal $z=w$.  This kernel $K$ emerged after lengthy but rather opaque calculations involving the quantities appearing in the previous section; it would be desirable to have a conceptual interpretation of this expression.  

We can now derive Theorem \ref{main} from the following three facts.

\begin{proposition}\label{ident} Formally at least, we have the following claims:
\begin{itemize}
\item[(i)] We have
\begin{equation}\label{max} \partial_k \Phi( k^{-1/2}_* \mu^{\boxplus k} )|_{k=1} = \frac{8\pi^2}{3} \int_\R f^3\ dx + 4\pi \int_\R \frac{Hf \partial_x f - f \partial_x Hf}{(Hf)^2 + f^2} f^2\ dx.
\end{equation}
\item[(ii)]  We have
\begin{equation}\label{sam}
\begin{split}
& \lim_{\eps \to 0^+} \sum_{\alpha,\beta \in \{-1,+1\}} \int_\R\int_\R f(x) f(y) K( x_{\alpha\eps}, y_{\beta \eps} )\ dx dy \\
&\quad = -\frac{8\pi^2}{3} \int_\R f^3\ dx - 4\pi \int_\R \frac{Hf \partial_x f - f \partial_x Hf}{(Hf)^2 + f^2} f^2\ dx,
\end{split}
\end{equation}
where $x_{\alpha \eps} \coloneqq x+i\alpha \eps$ and $y_{\beta \eps} \coloneqq y + i \beta \eps$.
\item[(iii)]  The kernel $K(z,\overline{w})$ is positive semi-definite, thus
$$ \sum_{j=1}^n \sum_{k=1}^n c_j \overline{c_k} K(z_j,\overline{z_k}) \geq 0$$
for all complex numbers $z_1,\dots,z_n \in \C \backslash \R$ and $c_1,\dots,c_n \in \C$.
\end{itemize}
\end{proposition}

Indeed, from (iii) we have
$$ \sum_{\alpha,\beta \in \{-1,+1\}} \int_\R\int_\R f(x) f(y) K( x_{\alpha \eps}, y_{\beta \eps} )\ dx dy
= \sum_{\alpha,\beta \in \{-1,+1\}} \int_\R\int_\R f(x) f(y) K( x+i \alpha \eps, \overline{y_{\beta \eps}} )\ dx dy$$
is non-negative for any $\eps > 0$.  Meanwhile, from (i), (ii) we have
$$ \lim_{\eps \to 0^+} \sum_{\alpha,\beta \in \{-1,+1\}} \int_\R\int_\R f(x) f(y) K( x_{\alpha \eps}, y_{\beta \eps} )\ dx dy 
= - \partial_k \Phi( k^{-1/2}_* \mu^{\boxplus k} )|_{k=1}$$
and hence
$$ \partial_k \Phi( k^{-1/2}_* \mu^{\boxplus k} )|_{k=1} \leq 0.$$
From \eqref{mono}, \eqref{ch-2} we then have
$$ \partial_k \Phi( k^{-1/2}_* \mu^{\boxplus k} ) \leq 0$$
for all $k \geq 1$, giving the non-increasing nature of $\Phi( k^{-1/2}_* \mu^{\boxplus k} )$, then the non-decreasing nature of $\chi( k^{-1/2}_* \mu^{\boxplus k} )$ follows from \eqref{tik} as in the previous section.

It remains to establish the three claims in Proposition \ref{ident}.  We begin with (i).  From \eqref{23-phi} and the chain rule we have
$$
\partial_k \Phi( k^{-1/2}_* \mu^{\boxplus k} )_{k=1} = 
4\pi^2 \int_\R f_1(x)^2 \partial_k f_k(x)|_{k=1}\ dx.$$
On the other hand, applying \eqref{diffeq-2} at $z=x+i0^+$ and using \eqref{plemelj} and the Cauchy-Riemann equations we have
$$ (k \partial_k + \frac{1}{2} x \partial_x) (\pi Hf_k(x) - i \pi f_k(x)) = \frac{ \partial_x (\pi Hf_k(x) - i \pi f_k(x))}{\pi Hf_k(x) - i \pi f_k(x)}  + \frac{1}{2} (\pi Hf_k(x) - i \pi f_k(x))$$
which on taking imaginary parts gives an integral differential equation for $f_k$:
\begin{equation}\label{feq}
 (k \partial_k + \frac{1}{2} x \partial_x) f_k = \frac{1}{\pi} \frac{Hf_k \partial_x f_k - f_k \partial_x Hf_k}{(Hf_k)^2 + f_k^2} + \frac{1}{2} f_k.
\end{equation}
Multiplying by $f_k^2$ and integrating, we obtain the claim (i) after a routine integration by parts.

We now skip ahead to (iii).  The Schur product theorem asserts that the pointwise product of positive semi-definite kernels is again positive semi-definite.  Since the rank one kernel $\frac{1}{G(z) G(\overline{w})}$ is clearly positive semi-definite, it thus suffices from \eqref{K-def} to show that the kernel
\begin{equation}\label{quant}
 \frac{G(z)-G(\overline{w})}{z-\overline{w}} + G(z) G(\overline{w}) 
\end{equation}
is negative semi-definite.  But from \eqref{G-form} and the identities
$$ \int_\R f(x)\ dx = 1$$
and
$$ -\frac{1}{(z-x)(\overline{w}-x)} = \frac{\frac{1}{z-x} - \frac{1}{\overline{w}-x}}{z-\overline{w}}$$
we see after a brief calculation that\footnote{In other words, the quantity \eqref{quant} is the negative of the covariance of $(z-X)^{-1}$ and $(\overline{w}-X)^{-1}$, where $X$ is a random variable with law $\mu$.}
$$ - \int_\R f(x) \left( \frac{1}{z-x} - G(z)\right) \left( \frac{1}{\overline{w}-x} - G(\overline{w})\right)\ dx = 
\frac{G(z)-G(\overline{w})}{z-\overline{w}} + G(z) G(\overline{w}).$$
Since $f(x)$ is non-negative and the rank one kernels $\left( \frac{1}{x-z} - G(z)\right) \left( \frac{1}{x-\overline{w}} - G(\overline{w})\right)$ are positive semi-definite, the claim (iii) follows.

It remains to establish the identity (ii), which is the lengthiest calculation.  We expand the left-hand side of \eqref{sam} as $A_2 + 2A_1 + A_0$, where
\begin{align*}
A_2 &\coloneqq  \lim_{\eps \to 0^+} \sum_{\alpha,\beta \in \{-1,+1\}} \int_\R\int_\R f(x) f(y) \frac{(G(x_{\alpha\eps})-G(y_{\beta \eps}))^2}{G(x_{\alpha \eps}) G(y_{\beta \eps}) (x_{\alpha \eps} - y_{\beta \eps})^2}\ dx dy \\
A_1 &\coloneqq  \lim_{\eps \to 0^+} \sum_{\alpha,\beta \in \{-1,+1\}} \int_\R\int_\R f(x) f(y) \frac{G(x_{\alpha\eps})-G(y_{\beta \eps})}{x_{\alpha \eps} - y_{\beta \eps}}\ dx dy \\
A_0 &\coloneqq  \lim_{\eps \to 0^+} \sum_{\alpha,\beta \in \{-1,+1\}} \int_\R\int_\R f(x) f(y) G(x_{\alpha \eps}) G(y_{\beta \eps})\ dx dy.
\end{align*}
The quantity $A_0$ is easiest to compute, as it factorizes as
$$ \left| \sum_\pm \int_\R f(x) G(x+i0^\pm)\ dx \right|^2.$$
Applying \eqref{plemelj} and Lemma \ref{hilb} we conclude that $A_0=0$.

Now we turn to $A_1$.  In order to compute the limit $\eps \to 0^+$ it will be convenient to use integration by parts to replace the divergent-looking factor $\frac{1}{x_{\alpha \eps} - y_{\beta \eps}}$ with a tamer singularity.  More precisely, we write
$$ \frac{1}{x_{\alpha \eps} - y_{\beta \eps}} = \partial_x \Log (x_{\alpha \eps} - y_{\beta \eps})$$
where we define the $\Log z$ away from the branch cut $(-\infty,0)$ to be the branch of the complex logarithm with imaginary part in $(-\pi,\pi)$, and on the branch cut $(-\infty,0)$ we define the averaged limiting value
$$ \Log(-x) \coloneqq \frac{1}{2} (\Log(-x+i0^+) + \Log(-x+i0^-)) = \log|x|.$$
The above identity breaks down when $x_{\alpha\eps}-y_{\beta \eps}$ vanishes, but this will not cause difficulty due to the vanishing of the numerator $G(x_{\alpha\eps})-G(y_{\beta \eps})$ in this case.  Integrating by parts, we conclude that
$$
A_1 = -\lim_{\eps \to 0^+} \sum_{\alpha,\beta \in \{-1,+1\}} \int_\R\int_\R \partial_x \left( f(x) f(y) (G(x_{\alpha\eps})-G(y_{\beta \eps})\right)
\Log (x_{\alpha \eps} - y_{\beta \eps})\ dx dy.$$
As $\eps \to 0^+$, the quantity $\Log (x_{\alpha \eps} - y_{\beta \eps})$ converges\footnote{Here the indicator function $1_{y>x}$ is defined to equal $1$ when $y>x$ and $0$ otherwise.}  to $\log |x-y| + i \pi 1_{y>x} \frac{\alpha-\beta}{2}$ for $x \neq y$, while from \eqref{plemelj} $G(x_{\alpha\eps})-G(y_{\beta \eps})$ converges to $\pi ( Hf(x)-Hf(y) - i \alpha f(x) + i \beta f(y) )$.  For $f$ sufficiently regular, we conclude that
$$ \sum_{\alpha,\beta \in \{-1,+1\}} \partial_x \left( f(x) f(y) (G(x_{\alpha\eps})-G(y_{\beta \eps}))\right)
\Log (x_{\alpha \eps} - y_{\beta \eps})$$
converges to
\begin{align*}
& 4\pi \partial_x \left( f(x) f(y) (Hf(x)-Hf(y)) \right) \log |x-y| \\
&\quad + 2 \pi^2 \partial_x \left(f(x)^2 f(y)\right) 1_{y>x} + 2\pi^2 \partial_x \left(f(x) f(y)^2\right) 1_{y>x}
\end{align*}
and hence
\begin{align*}
 A_1 =& -4\pi \int_\R \int_\R \partial_x ( f(x) f(y) (Hf(x)-Hf(y)) ) \log |x-y|\ dx dy\\
&\quad  - 2 \pi^2 \int_\R \int_\R \partial_x (f(x)^2 f(y) + f(x) f(y)^2) 1_{y>x}\ dx dy.
\end{align*}
Integrating by parts, we conclude that
$$ A_1 = 4\pi \int_\R \int_\R \frac{f(x) f(y) (Hf(x)-Hf(y))}{x-y} \ dx dy - 4 \pi^2 \int_\R f^3\ dx.$$
By symmetry and \eqref{H-def} we have
$$ \int_\R \int_\R \frac{f(x) f(y) (Hf(x)-Hf(y))}{x-y} \ dx dy  = \pi \int_\R f(x) Hf(x) Hf(x)\ dx + \pi \int_\R Hf(y) f(y) Hf(y)\ dy $$
and hence by Lemma \ref{hilb} and a brief calculation
$$ A_1 = -\frac{4\pi^2}{3} \int_\R f^3\ dx.$$
We can compute $A_2$ in a similar fashion, writing
$$ \frac{1}{(x_{\alpha \eps} - y_{\beta \eps})^2} = -\partial_x^2 \Log (x_{\alpha \eps} - y_{\beta \eps})$$
and integrating by parts twice to obtain
$$ A_2 = - \lim_{\eps \to 0^+} \sum_{\alpha,\beta \in \{-1,+1\}} \int_\R\int_\R \partial_x^2 \left( f(x) f(y) \frac{(G(x_{\alpha\eps})-G(y_{\beta \eps}))^2}{G(x_{\alpha \eps}) G(y_{\beta \eps})}\right) \Log (x_{\alpha \eps} - y_{\beta \eps})\ dx dy.$$
We expand
\begin{align*}
 \frac{(G(x_{\alpha\eps})-G(y_{\beta \eps}))^2}{G(x_{\alpha \eps}) G(y_{\beta \eps})}
&= \frac{G(x_{\alpha \eps})}{G(y_{\beta \eps})} - 2 + \frac{G(y_{\beta \eps})}{G(x_{\alpha \eps})}  \\
&= \frac{G(x_{\alpha \eps}) \overline{G(y_{\beta \eps})}}{|G(y_{\beta \eps})|^2} - 2 +
\frac{G(y_{\beta \eps}) \overline{G(x_{\alpha \eps})}}{|G(x_{\alpha \eps})|^2} 
\end{align*}
and hence by \eqref{plemelj} this quantity converges to
$$
\frac{(Hf(x) - i \alpha f(x)) (Hf(y)+i\beta f(y))}{Hf(y)^2 + f(y)^2} - 2 + \frac{(Hf(y) - i \beta f(y)) (Hf(x)+i\alpha f(x))}{Hf(x)^2 + f(x)^2} $$
as $\eps \to 0^+$.  We can then evaluate $A_2$ as before (for $f$ sufficiently regular) as
\begin{align*}
 A_2 =& -4 \int_\R \int_\R \partial_x^2 \left( f(x) f(y) \left(\frac{Hf(x)Hf(y)}{Hf(y)^2 + f(y)^2} -2 + 
\frac{Hf(x)Hf(y)}{Hf(x)^2 + f(x)^2}\right) \right) \log|x-y|\ dx dy \\
&\quad - 2\pi \int_\R \int_\R \partial_x^2 \left( f(x) f(y) \left(\frac{f(x)Hf(y)}{Hf(y)^2 + f(y)^2} - 
\frac{f(x) Hf(y)}{Hf(x)^2 + f(x)^2}\right) \right) 1_{y>x}\ dx dy \\ 
&\quad - 2\pi \int_\R \int_\R \partial_x^2 \left( f(x) f(y) \left(\frac{f(y)Hf(x)}{Hf(y)^2 + f(y)^2} - 
\frac{f(y) Hf(x)}{Hf(x)^2 + f(x)^2}\right) \right) 1_{y>x}\ dx dy.
\end{align*}
From the fundamental theorem of calculus we have
$$
\int_\R \int_\R \partial_x^2 \left( f(x) f(y) \frac{f(x)Hf(y)}{Hf(y)^2 + f(y)^2} \right)1_{y>x}\ dx dy 
= \int_\R \partial_y (f(y)^2) \frac{f(y) Hf(y)}{Hf(y)^2 + f(y)^2}\ dy.$$
Using the distributional identity $\partial_x^2 1_{y>x} = \partial_y^2 1_{y>x}$ and integrating by parts repeatedly, we also have
\begin{align*}
&\int_\R \int_\R \partial_x^2 \left( f(x) f(y) \frac{f(x)Hf(y)}{Hf(x)^2 + f(x)^2} \right)1_{y>x}\ dx dy \\
&\quad = 
\int_\R \int_\R \partial_y^2 \left( f(x) f(y) \frac{f(x)Hf(y)}{Hf(x)^2 + f(x)^2} \right)1_{y>x}\ dx dy \\
&\quad = -\int_\R \partial_x (f(x) Hf(x)) \frac{f(x)^2}{Hf(x)^2 + f(x)^2}\ dx.
\end{align*}
Similar computations give
$$ \int_\R \int_\R \partial_x^2 \left( f(x) f(y) \frac{f(y)Hf(x)}{Hf(y)^2 + f(y)^2} \right) 1_{y>x}\ dx dy
= \int_\R \partial_y (f(y) Hf(y)) \frac{f(y)^2}{Hf(y)^2 + f(y)^2}\ dy $$
and
$$ \int_\R \int_\R \partial_x^2 \left( f(x) f(y) \frac{f(y) Hf(x)}{Hf(x)^2 + f(x)^2} \right) 1_{y>x}\ dx dy
= - \int_\R \partial_x (f(x)^2) \frac{f(x) Hf(x)}{Hf(x)^2 + f(x)^2}\ dx.$$
Next, we integrate by parts, then use Lemma \ref{hilb} and the fact that $H$ commutes with derivatives to compute
\begin{align*}
& \int_\R \int_\R \partial_x^2\left( f(x) f(y) \frac{Hf(x)Hf(y)}{Hf(y)^2 + f(y)^2} \right) \log|x-y|\ dx dy \\
&\quad = - \int_\R \left(\mathrm{p.v.} \int_\R \partial_x( f(x) f(y) )\frac{Hf(x)Hf(y)}{Hf(y)^2 + f(y)^2}\ \frac{dx}{x-y}\right) dy \\
&\quad = \pi \int_\R H \partial_y(f Hf)(y) \frac{f(y) Hf(y)}{Hf(y)^2 + f(y)^2} dy \\
&\quad = \frac{\pi}{2} \int_\R \partial_y((Hf)^2-f^2)(y) \frac{f(y) Hf(y)}{Hf(y)^2 + f(y)^2} dy.
\end{align*}
From $\partial_x^2 \log |x-y| = \partial_y^2 \log |x-y|$, integration by parts, and symmetry we then have
\begin{align*}
& \int_\R \int_\R \partial_x^2\left( f(x) f(y) \frac{Hf(x)Hf(y)}{Hf(x)^2 + f(x)^2} \right) \log|x-y|\ dx dy \\
&\quad=  \int_\R \int_\R \partial_y^2\left( f(x) f(y) \frac{Hf(x)Hf(y)}{Hf(x)^2 + f(x)^2} \right) \log|x-y|\ dx dy \\
&\quad = \frac{\pi}{2} \int_\R \partial_x((Hf)^2-f^2)(x) \frac{f(x) Hf(x)}{Hf(x)^2 + f(x)^2} dx.
\end{align*}
Finally,
\begin{align*}
 \int_\R \int_\R \partial_x^2( f(x) f(y)  ) \log|x-y|\ dx dy 
&= -  \int_\R \left(\mathrm{p.v.} \int_\R \partial_x( f(x) f(y)  )\ \frac{dx}{x-y}\right) dy \\
&= \pi \int_\R H \partial_y f(y) f(y)\ dy.
\end{align*}
Putting all this together, we conclude that
\begin{align*}
A_2 &= 2\pi\int_\R - \partial_x( (Hf)^2-f^2 ) \frac{f Hf}{(Hf)^2+f^2} + 4 (H \partial_x f) f - \partial_x( (Hf)^2-f^2 ) \frac{f Hf}{(Hf)^2+f^2}\ dx \\
&\quad - 2\pi \int_\R \partial_x (f^2) \frac{f Hf}{(Hf)^2 + f^2} + \partial_x (f Hf) \frac{f^2}{(Hf)^2 + f^2}\ dx \\
&\quad - 2\pi \int_\R \partial_x (f^2) \frac{f Hf}{(Hf)^2 + f^2} + \partial_x (f Hf) \frac{f^2}{(Hf)^2 + f^2}\ dx 
\end{align*}
which on applying the Leibniz rule, the commutativity of $H$ and $\partial_x$, and collecting terms, simplifies to
$$
A_2 = -4\pi \int_\R \frac{H f \partial_x f - f \partial_x Hf}{(Hf)^2 + f^2}\ f^2\ dx$$
and the claim (ii) follows.

\section{Variational formulation}\label{variation-sec}

We now prove Theorem \ref{var-thm}.  Our calculations here will be completely formal. Similar calculations appear in the recent paper \cite[\S 5]{kenyon} in the context of studying random Young tableaux from a variational perspective; we thank Istvan Prause for this reference.

We assume that the measures $\mu^{\boxplus k}$ are absolutely continuous with 
$$ d\mu^{\boxplus k} = f_k(x)\ dx$$
for $k \geq 1$.  Applying \eqref{plemelj} at $x+i0^+$ together with \eqref{diffeq}, we conclude that
$$
(k \partial_k + x \partial_x) (\pi H f_k - \pi i f_k) = \frac{\partial_x (Hf_k - if_k)}{H f_k - if_k} = \partial_x \log(Hf_k - if_k)$$
and hence on taking real and imaginary parts we have
$$
(k \partial_k + x \partial_x) Hf_k = \frac{1}{\pi} \partial_x \log ((Hf_k)^2 + f_k^2)^{1/2} $$
and
$$
(k \partial_k + x \partial_x) f_k = \frac{1}{\pi} \partial_x \arctan \frac{f_k}{Hf_k}$$
(where we use the branch of $\arctan$ taking values in $[0,\pi]$) and thus by the change of variables $k=1/s$ and abbreviating $f \coloneqq f_{1/s}$,
\begin{equation}\label{sy}
(-s \partial_s + x \partial_x) Hf = \frac{1}{\pi} \partial_x \log((Hf)^2+f^2)^{1/2}
\end{equation}
and
\begin{equation}\label{sx}
(-s \partial_s + x \partial_x) f = \frac{1}{\pi} \partial_x \arctan \frac{f}{Hf}
\end{equation}
for $0 < s < 1$.  We remark that this latter equation was also formally derived in \cite{steinerberger} in the context of the derivative process, which is an averaged version of the minor process as established in \cite[Lemma 1.16]{mss}.

Meanwhile, for $(s,y) \in \Delta$, we have from \eqref{mus} that
\begin{equation}\label{help}
\int_{-\infty}^{\lambda/s} f(x)\ dx = \frac{y}{s},
\end{equation}
where we abbreviate $\lambda = \lambda(s,y)$ and $f = f_{1/s}$.  If we differentiate this in $y$ using the fundamental theorem of calculus, we see that
$$ \frac{\partial_y \lambda}{s}  f(\lambda/s) = \frac{1}{s}$$
thus
\begin{equation}\label{fls}
 f(\lambda/s) = \frac{1}{\partial_y \lambda}.
\end{equation}
If instead we differentiate in $s$, we conclude that
$$ \left(\frac{\partial_s \lambda}{s} - \frac{\lambda}{s^2}\right) f(\lambda/s) + \int_{-\infty}^{\lambda/s} \partial_s f(x)\ dx = - \frac{y}{s^2}.$$
and thus by \eqref{fls} and multiplying by $s$
$$ \frac{\partial_s \lambda}{\partial_y \lambda} - \frac{\lambda}{s \partial_y \lambda} + \int_{-\infty}^{\lambda/s} s\partial_s f(x)\ dx = - \frac{y}{s}.$$
By \eqref{sx} we have
$$
 s \partial_s f = x \partial_x f - \frac{1}{\pi} \partial_x \arctan \frac{f}{Hf} 
$$
and hence by integration by parts and \eqref{help}, \eqref{fls}
\begin{align*}
\int_{-\infty}^{\lambda/s} s\partial_s f(x)\ dx &= \frac{\lambda}{s} f(\lambda/s) - \frac{1}{\pi} \arctan \frac{f(\lambda/s)}{Hf(\lambda/s)}
- \int_{-\infty}^{\lambda/s} f(x)\ dx \\
&= \frac{\lambda}{s \partial_y \lambda} - \frac{1}{\pi} \arctan \frac{f(\lambda/s)}{Hf(\lambda/s)} - \frac{y}{s}
\end{align*}
and thus
$$ \frac{\partial_s \lambda}{\partial_y \lambda} = \frac{1}{\pi} \arctan \frac{f(\lambda/s)}{Hf(\lambda/s)}.$$
We remark that this gives the pointwise inequalities $0 \leq \partial_s \lambda \leq \partial_y \lambda$, which in the random matrix formulation corresponds to the Cauchy interlacing inequalities.  We rewrite this equation using \eqref{fls} as
\begin{equation}\label{hfs}
 Hf(\lambda/s) = \frac{\cot(\pi \frac{\partial_s \lambda}{\partial_y \lambda})}{\partial_y \lambda} 
\end{equation}
and hence
$$ \log ((Hf)^2+f^2)^{1/2}(\lambda/s) = \log \frac{\cosec(\pi \frac{\partial_s \lambda}{\partial_y \lambda})}{\partial_y \lambda}.$$
Differentiating in $y$ using the chain rule, we conclude
$$ \frac{\partial_y \lambda}{s} \left(\partial_x \log ((Hf)^2+f^2)^{1/2}\right)(\lambda/s)
= \partial_y \log \frac{\cosec(\pi \frac{\partial_s \lambda}{\partial_y \lambda})}{\partial_y \lambda}$$
and similarly by differentiating \eqref{hfs} in $y$, $s$ we have
$$ \frac{\partial_y \lambda}{s} (\partial_x Hf)(\lambda/s) = \partial_y \frac{\cot(\pi \frac{\partial_s \lambda}{\partial_y \lambda})}{\partial_y \lambda}$$
and
$$ (\partial_s Hf)(\lambda/s) + \left(\frac{\partial_s \lambda}{s}-\frac{\lambda}{s^2}\right) (\partial_x Hf)(\lambda/s)
= \partial_s \frac{\cot(\pi \frac{\partial_s \lambda}{\partial_y \lambda})}{\partial_y \lambda}
$$
so that
$$ (\partial_s Hf)(\lambda/s) 
= \left(\partial_s - \frac{\partial_s \lambda}{\partial_y \lambda} \partial_y + \frac{\lambda}{s\partial_y \lambda} \partial_y\right) \frac{\cot(\pi \frac{\partial_s \lambda}{\partial_y \lambda})}{\partial_y \lambda}.
$$
Inserting these identities into \eqref{sy} evaluated at $\lambda/s$, we obtain a differential equation for $\lambda$ in the variables $s,y$:
$$ 
-s\left(\partial_s - \frac{\partial_s \lambda}{\partial_y \lambda} \partial_y\right) \frac{\cot(\pi \frac{\partial_s \lambda}{\partial_y \lambda})}{\partial_y \lambda}
= \frac{1}{\pi} \frac{s}{\partial_y \lambda} \partial_y \log \frac{\cosec(\pi \frac{\partial_s \lambda}{\partial_y \lambda})}{\partial_y \lambda}.$$
Multiplying by $-\pi \partial_y \lambda/s$, we obtain
$$ 
(\partial_y \lambda \partial_s - \partial_s \lambda \partial_y) \frac{\pi \cot(\pi \frac{\partial_s \lambda}{\partial_y \lambda})}{\partial_y \lambda}
= \partial_y \left( \log \partial_y \lambda + \log \sin(\pi \frac{\partial_s \lambda}{\partial_y \lambda}) \right)$$
which we write in divergence form using \eqref{L-def} as
$$ \partial_s\left( \partial_y \lambda \frac{\pi \cot(\pi \frac{\partial_s \lambda}{\partial_y \lambda})}{\partial_y \lambda} \right)
+ \partial_y\left( - \partial_s \lambda \frac{\pi \cot(\pi \frac{\partial_s \lambda}{\partial_y \lambda})}{\partial_y \lambda}
- L(\partial_s \lambda, \partial_y \lambda)\right) = 0.$$
Since the partial derivatives of
$$ L(\lambda_s, \lambda_y) \coloneqq \log \lambda_y + \log \sin\left(\pi \frac{\lambda_s}{\lambda_y}\right)$$
are given by
$$ L_{\lambda_s} = \frac{\pi}{\lambda_y} \cot\left(\pi \frac{\lambda_s}{\lambda_y}\right)$$
and
$$ L_{\lambda_y} = \frac{1}{\lambda_y} - \frac{\pi \lambda_s}{\lambda_y^2} \cot\left(\pi \frac{\lambda_s}{\lambda_y}\right),$$
we can rewrite the above equation as
$$ \partial_s (\partial_y \lambda L_{\lambda_s}(\partial_s \lambda, \partial_y \lambda)) + \partial_y (\partial_y \lambda L_{\lambda_y}(\partial_s \lambda, \partial_y \lambda) - L(\partial_s \lambda, \partial_y \lambda)) = 0.$$
From the chain rule we have
$$ \partial_y L(\partial_s \lambda, \partial_y \lambda) = (\partial_y \partial_y \lambda) L_{\lambda_y}( \partial_s \lambda, \partial_y \lambda)
+ (\partial_s \partial_y \lambda) L_{\lambda_s}( \partial_s \lambda, \partial_y \lambda);$$
inserting this into the previous equation and using the product rule and then cancelling the $\partial_y \lambda$ factor, we conclude that
$$ \partial_s L_{\lambda_s}(\partial_s \lambda, \partial_y \lambda) + \partial_y L_{\lambda_y}(\partial_s \lambda, \partial_y \lambda) = 0$$
which is the Euler-Lagrange equation for the Lagrangian \eqref{lag}, and the claim follows.

\appendix

\section{Fractional free convolution powers from the minor process}\label{power}

In this appendix we prove Proposition \ref{minor-free}.  Let the hypotheses be as in that proposition; our task is to establish \eqref{pipx}.  We follow the arguments from \cite[\S 2.5.4]{tao-book}.  Using the GNS construction we may assume that ${\mathcal A}$ is a von Neumann algebra of bounded operators.

We begin with some algebraic identities.  For any noncommutative variable $E$ of operator norm less than $1$, define the transform
$$ \Psi(E) \coloneqq (1-E)^{-1} - 1 = E + E^2 + E^3 + \dots.$$

\begin{lemma}[Algebraic identities]\label{alg}\  
\begin{itemize}
\item[(i)]  If $Z \in \A$ is sufficiently small (in operator norm), then
$$ (1 + \pi(Z))^{-1} = \pi( (1 + pZ)^{-1} ).$$
\item[(ii)]  If $Y \in \A$, and $E \in \A$ is sufficiently small (in operator norm) depending on $Y$, then
$$ \pi\left(Y \Psi(E)\right) = \Psi(E_{pY}),$$
where 
\begin{equation}\label{ep}
 E_{pY} \coloneqq \pi\left( 1 - (1-E) (1-(1-pY)E)^{-1} \right).
\end{equation}
\end{itemize}
\end{lemma}

\begin{proof}  If $Z$ is small enough, then $1+pZ$ is invertible by Neumann series, and by \eqref{piy} we have
\begin{align*} 
\pi( (1+pZ)^{-1} ) (1+\pi(Z)) &= \pi( (1+pZ)^{-1} p (1+Z) p) \\
& = \pi( (1+pZ)^{-1} (1+pZ) p ) \\
&= \pi(p) \\
&= 1
\end{align*} 
giving (i).  For (ii), we apply (i) with $Z \coloneqq Y \Psi(E)$ to conclude that
$$ \left(1 + \pi( Y \Psi(E) )\right)^{-1} = \pi\left( \left(1 + p Y \Psi(E) \right)^{-1} \right).$$
Since
$$1 + pY \Psi(E) = (1-(1-pY)E) (1-E)^{-1}$$
we see from \eqref{ep} that
$$ \pi\left( \left(1 + pY \Psi(E) \right)^{-1} \right) = 1 - E_{pY}$$
and the claim (ii) then follows after some rearranging.
\end{proof}

Now set $Y=k$.  From \eqref{ep} and Neumann series we have
$$ E_{kp} = [pEp] - \sum_{n=1}^\infty [p((1-kp)E)^np] - [pE ((1-kp)E)^np] $$
when $E$ is sufficiently small in operator norm.  As $1-kp$ has trace zero, we conclude on taking traces that
\begin{equation}\label{Taup}
 \tau_p(E_{kp}) = 0
\end{equation}
whenever $E$ has trace zero, is sufficiently small in operator norm, and is freely independent from $p$.

This has the following consequence.  If $z$ is sufficiently large and $s = G_\mu(z)$, then from \eqref{gdef} we have
$$ s = \tau( (z - X)^{-1} )$$
and thus 
\begin{equation}\label{zae}
 (z-X)^{-1} = s ( 1 - E(s) )
\end{equation}
for some trace zero element $E(s) \in \A$, which will be small when $z$ is large.  Since $X$ is freely independent of $p$, $E(s)$ is also.  Meanwhile from \eqref{rg} one has
$$
R_\mu(s) + \frac{1}{s} = z$$
which when combined with \eqref{zae} and rearranging gives
$$ X = R_\mu(s) - \frac{1}{s} \Psi(E(s))$$
for all sufficiently small $s$.  Applying $k \pi$, we conclude that
$$ k \pi(X) = k R_\mu(s) - \frac{1}{s} \pi(k \Psi(E(s))).$$
Applying Lemma \ref{alg}(ii) and \eqref{taup}, we conclude that
$$ k \pi(X) = k R_\mu(s) - \frac{1}{s} \Psi(E_{kp}(s))$$
where $E_{kp}(s) \in \A_p$ obeys \eqref{taup}.  If we set $z' \coloneqq k R_\mu(s) + \frac{1}{s}$, we can rearrange this as
$$ (z' - k \pi(X))^{-1} = s (1 - E_{kp}(s))$$
and then on taking traces we conclude that
$$ s = \tau_p( (z' - k \pi(X))^{-1} ) = G_{k \pi(X)}(z).$$
From \eqref{rg} we then conclude that $R_{k\pi(X)}(s) = k R_\mu(s)$ for all sufficiently small $s$, giving the claim \eqref{pipx}.

\section{Monotonicity for microstates free entropy} \label{sec-microstates}

\begin{center}
	by David Jekel
\end{center}

In this section, we adapt the free probability proof of Theorem \ref{sev-var} to the microstates setting to obtain an analog of that theorem for Voiculescu's \emph{microstates free entropy} $\chi$, introduced in \cite{voi-entropy2}.  The main result is as follows.

\begin{theorem}[Monotonicity of microstate free entropy] \label{microstates}
Let $k \in [1,\infty)$.  Let $(\A,\tau)$ be a noncommutative probability space, let $X \in \A_{\sa}^n$ (i.e., $X$ is a tuple $(X_1,\dots,X_n)$ of self-adjoint elements of $\A$), and let $p$ be a projection of trace $1/k$ freely independent of $X$.  Let $\pi \colon \A \to [p\A p]$ be the compression map, and let $\Pi \coloneqq k^{1/2} \pi$ be the normalized compression.  Then $\chi(\Pi(X)) \geq \chi(X)$.
\end{theorem}

The first step in the proof is to reformulate $\chi$ in terms of the classical entropy of random matrix approximations of $X$.  The second step is to apply a similar argument as in \S \ref{score-proof} for the \emph{classical} entropy and score functions, which results in an approximate version of \eqref{ko} for the minors of the random matrix models.

We first set up all the notation that we need.  We begin by recalling various classical information theory notions in the general context of random variables taking values in finite dimensional inner product spaces\footnote{All inner product spaces here will be over the reals.} $H$.

\begin{definition}[Classical information theory concepts]  Let $H$ be a finite dimensional inner product space, with inner product $\ip{u,v}_H$ and norm $\norm{u}_H$.  We let $\nu_H$ be the Haar measure canonically associated to $H$ (thus $\nu$ assigns unit mass to the unit cube generated by any orthonormal basis in $H$).
\begin{itemize}
\item[(i)]  If $X$ is square integrable, the \emph{total variance} $\Var_H(X)$ is given by the formula
$$\Var_H(X) \coloneqq \mathbb{E} \norm{X - \mathbb{E}X}_H^2.$$ 
\item[(ii)]  If $X$ is a (classical) random variable taking values in $H$ with absolutely continuous law $d\mu = \rho\ d\nu_H$, the \emph{classical (differential) entropy} $h_H(X)$ of $X$ is given by the formula
$$ -\int_{H} \rho \log \rho\ d\nu_H.$$
If there is no density $\rho$, the entropy is defined to equal $-\infty$. We also write $h_H(\mu)$ for $h_H(X)$.
\item[(iii)]  A \emph{standard gaussian random variable} in $H$ is a gaussian variable $Z_H$
 of mean zero and identity covariance matrix in the sense that
$$\mathbb{E} \ip{ u, Z_H }_H \ip{ Z_H, v }_H = \langle u, v \rangle_H$$ 
for all $u,v \in H$; equivalently, $Z_H$ has law $(2\pi)^{-\mathrm{dim}(H)/2} e^{-\norm{u}_H^2/2}\ d\nu_H$.  
\item[(iv)] If $X$ is a random variable taking values in $H$, then a random variable $J_H(X)$ is said to be a \emph{classical score} of $X$ (relative to the inner product $H$) if it lies in the $L^2$ closure of the algebra generated by $X$, and
\begin{equation} \label{class-score}
 \frac{d}{d\eps} \mathbb{E} \ip{ f(X + \eps Z_H), Z_H }_H|_{\eps = 0} = \mathbb{E} \ip{J_H(X), f(X) }_H
\end{equation}
for any $f \in C_c^\infty(H;H)$, where $Z_H$ is a standard gaussian variable in $H$ (classically) independent of $X$; compare with \eqref{tautau}. Note that if the classical score exists, it is unique.
\item[(v)] The \emph{classical Fisher information} of $X$ is $\mathcal{I}_H(X) \coloneqq \mathbb{E} \norm{J_H(X)}_H^2$ if a classical score $J_H(X)$ exists, and $\mathcal{I}_H(X) = +\infty$ otherwise.
\end{itemize}
\end{definition}

\begin{example}  If $H$ is a standard Euclidean space $\R^d$, and $X$ has a $C^1$ probability density $\rho$, then the classical score is given explicitly by $J_{\R^d}(X) = -\nabla \rho(X) / \rho(X)$ provided the latter is in $L^2$.  The classical Fisher information is then equal to $\mathcal{I}_{\R^d}(X) = \int_{\R^d} \frac{|\nabla \rho|^2}{\rho}$.
\end{example} 

\begin{example}\label{gauss-standard}  If $H$ is a $d$-dimensional Hilbert space, $Z_H$ is a standard gaussian variable in $H$, and $t>0$, then $\Var_H(t^{1/2} Z_H) = td$, $h_H(t^{1/2}Z_H) = \frac{d}{2} \log(2\pi e t)$, $J_H(t^{1/2} Z_H) = t^{-1/2} Z_H$, and $\mathcal{I}_H(t^{1/2} Z_H) = \frac{d}{t}$.  Thus we see that with this ``standard'' choice of normalization, most quantities scale linearly with the dimension $d$.  Later on we shall switch to a ``microstate'' choice of normalization that is better suited for passing to the free probability limit $d \to \infty$.
\end{example}

We now recall some standard properties of the above notions:

\begin{lemma}[Standard classical information theory facts]\label{prob-standard}  Let $H$ be a finite dimensional inner product space of some dimension $d$, with canonical Haar measure $\nu_H$. Let $X$ be a random variable taking values in $H$ with law $\mu$, and let $Z_H$ be a standard gaussian random variable in $H$ classically independent of $X$.
\begin{itemize}
\item[(i)]  (Entropy controlled by variance)  If $X$ has finite variance, then
\begin{equation}\label{entropy-max}
-\infty < h_H(X) \leq \frac{d}{2} \log(2\pi e \Var_H(X) / d).
\end{equation}
In particular, each multiple $t^{1/2} Z_H$ of $Z_H$ maximizes the entropy amongst all variables of the same variance.  
\item[(ii)]  (Entropy controlled by partition) Let $(S_j)_{j=1}^\infty$ be a measurable partition of $H$.  Then
\[
h_H(\mu) \leq \sum_{j=0}^\infty \mu(S_j) \log \nu_H(S_j) - \sum_{j=0}^\infty \mu(S_j) \log \mu(S_j)
\]
\item[(iii)]  (Shannon inequality) If $Y$ is a random variable in $H$ classically independent of $X$, then $h_H(X+Y) \geq h_H(X), h_H(Y)$.
\item[(iv)]  (Stein identity) If $t>0$, then the score $J_H(X + t^{1/2} Z_H)$ exists and is given by the formula
\begin{equation}\label{stein-ident}
J_H(X + t^{1/2} Z_H) = \mathbb{E}[t^{-1/2} Z_H | X + t^{1/2} Z_H].
\end{equation}
In particular the Fisher information $\mathcal{I}_H(X + t^{1/2}Z_H)$ is finite.
\item[(v)] (de Bruijn identity)  If $0 < t_0 < t_1$, we have the identity
\begin{equation}\label{debruijn-ident}
h_H(X + t_1^{1/2}Z_H) - h_H(X + t_0^{1/2} Z_H) = \frac{1}{2} \int_{t_0}^{t_1} \mathcal{I}_H(X + t^{1/2}Z_H)\,dt.
\end{equation}
\end{itemize}
\end{lemma}

\begin{proof}  By using an orthonormal basis one can identify $H$ with a standard Euclidean space $\R^d$.  The facts (i), (iv), (v) are then well known and can be found for instance in \cite{stam}, and (iii) is similarly well known \cite{shannon}.  Now we prove (ii). If $\mu$ does not have a density, then $h_H(\mu) = -\infty$ and hence the claim is trivially true.  Assume that $\mu$ has a density $\rho$.  Then
\[
h_H(\mu) = -\sum_{j=0}^\infty \int_{S_j} \rho \log \rho\ d\nu_H.
\]
We apply Jensen's inequality to the concave function $-t \log t$ and the probability measure that is the push-forward by $\rho$ of the uniform distribution on $S_j$, and thus obtain
\begin{align*}
\frac{1}{\nu_H(S_j)} \int_{S_j} - \rho \log \rho\,dx &\leq -\left( \frac{1}{\nu_H(S_j)} \int_{S_j} \rho\ d\nu_H \right) \log \left( \frac{1}{\nu_H(S_j)} \int_{S_j} \rho\ d\nu_H\right) \\
&= -\frac{\mu(S_j)}{\nu_H(S_j)} \log \frac{\mu(S_j)}{\nu_H(S_j)},
\end{align*}
which produces the desired estimate.
\end{proof}

We will primarily work in the inner product space $M_N(\C)_{\sa}^n$ of $n$-tuples $X = (X_1,\dots,X_N)$ of $N \times N$ Hermitian matrices, with 
inner product
\[
\ip{X,Y}_{M_N(\C)_{\sa}^n} \coloneqq \sum_{j=1}^n \tr_N(X_j Y_j)
\]
defined using the normalized trace
$$ \tr_N \coloneqq \frac{1}{N} \Tr,$$
thus in particular we have the normalized Frobenius norms
\[
\norm{X}_{M_N(\C)_{\sa}^n}^2 = \sum_{j=1}^n \tr_N(X_j^2) = \frac{1}{N} \sum_{j=1}^n \Tr(X_j^2).
\]
This is an $nN^2$-dimensional inner product space.  If $Z_{M_N(\C)_{\sa}^n}$ is a standard gaussian random variable in $M_N(\C)_{\sa}^n$, then $Z_{M_N(\C)_{\sa}^n}$ is an ensemble of $n$ (classically) independent matrices, with each entry having variance $N$, for a total variance of $\Var_{M_N(\C)_{\sa}^n}(Z) = nN^2$.  To facilitate taking limits as $N \to \infty$, it is convenient to introduce the normalized gaussian variable
$$ Z^{(N)} \coloneqq \frac{1}{N} Z_{M_N(\C)_{\sa}^n}.$$
Thus $Z^{(N)}$ is an ensemble of $n$ (classically) independent GUE matrices, with each entry having variance $1/N$, converging to an $n$-tuple of freely independent semicircular random variables as $N \to \infty$ \cite{voi-limit-law}; we refer to such random variables $Z^{(N)}$ as \emph{GUE tuples} in $M_N(\C)_{\sa}^n$.  One easily computes the total variance
\begin{equation}\label{total-var}
\Var_{M_N(\C)_{\sa}^n}(t^{1/2} Z^{(N)}) = t n,
\end{equation}
classical entropy
\begin{equation}\label{entropies}
h_{M_N(\C)_{\sa}^n}(t^{1/2} Z^{(N)}) = \frac{nN^2}{2} \log(2\pi et) - nN^2 \log N,
\end{equation}
classical score
\begin{equation}\label{score}
J_{M_N(\C)_{\sa}^n}(t^{1/2} Z^{(N)}) = N^2 t^{-1/2} Z^{(N)},
\end{equation}
and classical Fisher information
\begin{equation}\label{fisher}
\mathcal{I}_{M_N(\C)_{\sa}^n}(t^{1/2} Z^{(N)}) = \frac{nN^4}{t}
\end{equation}
of multiples $t^{1/2} Z^{(N)}$ of GUE tuples for $t>0$.  Note that most of the quantities on the right-hand side depend on the matrix dimension $N$, which is undesirable for the purposes of extracting a meaningful limit as $N \to \infty$.  To facilitate the process of taking such a limit, we therefore introduce the \emph{normalized classical entropy}
$$ h^{(N)}(X) \coloneqq \frac{1}{N^2} h_{M_N(\C)_{\sa}^n}(X) + n \log N,$$
the \emph{normalized classical score}
$$ J^{(N)}(X) \coloneqq \frac{1}{N^2} J_{M_N(\C)_{\sa}^n}(X),$$
and the \emph{normalized classical Fisher information}
$$ \mathcal{I}^{(N)}(X) \coloneqq \mathbb{E} \norm{J^{(N)}(X)}_H^2 = \frac{1}{N^4} \mathcal{I}_{M_N(\C)_{\sa}^n}(X)$$
while leaving the variance unchanged:
$$ \Var^{(N)}(X) \coloneqq \Var_{M_N(\C)_{\sa}^n}(X).$$
Thus for instance we have
\begin{equation}\label{gauss-val}
\begin{split}
\Var^{(N)}(t^{1/2} Z^{(N)}) &= tn \\
h^{(N)}( t^{1/2} Z^{(N)} ) &= \frac{n}{2} \log(2\pi e t) \\
J^{(N)}( t^{1/2} Z^{(N)} ) &= t^{-1/2} Z^{(N)} \\
\mathcal{I}^{(N)}( t^{1/2} Z^{(N)} ) &= \frac{n}{t}.
\end{split}
\end{equation}
Comparing this with Example \ref{gauss-standard}, we see that these normalizations have lowered the ``effective dimension'' of $M_N(\C)_{\sa}^n$ from $nN^2$ to $n$.  With these ``microstate'' normalizations, the definition \eqref{class-score} of the classical score becomes
\begin{equation} \label{class-score-norm}
 \frac{d}{d\eps} \mathbb{E} \ip{ f(X + \eps Z^{(N)}), Z^{(N)} }_{M_N(\C)_{\sa}^n}|_{\eps = 0} = \mathbb{E} \ip{J^{(N)}(X), f(X) }_{M_N(\C)_{\sa}^n}
\end{equation}
the relationship \eqref{entropy-max} between classical entropy and variance becomes
\begin{equation}\label{entropy-max-norm}
-\infty < h^{(N)}(X) \leq \frac{n}{2} \log(2\pi e \Var^{(N)}(X) / n),
\end{equation}
the Stein identity \eqref{stein-ident} becomes
\begin{equation}\label{stein-ident-norm}
J^{(N)}(X + t^{1/2} Z^{(N)}) = \mathbb{E}[t^{-1/2} Z^{(N)} | X + t^{1/2} Z^{(N)}]
\end{equation}
and the de Bruijn identity \eqref{debruijn-ident} becomes
\begin{equation}\label{debruijn-ident-norm}
h^{(N)}(X + t_1^{1/2}Z^{(N)}) - h^{(N)}(X + t_0^{1/2} Z^{(N)}) = \frac{1}{2} \int_{t_0}^{t_1} \mathcal{I}^{(N)}(X + t^{1/2}Z^{(N)})\,dt.
\end{equation}
Note how there are no longer any factors of $N$ appearing explicitly in these assertions (other than in the superscripts and subscripts).  The reader is invited to verify that these identities and inequalities are compatible with \eqref{gauss-val}. See, e.g., \cite[\S16]{JekelThesis} for further explanation of these normalizations.

Now we introduce the definitions necessary to define microstate entropy.

\begin{definition}[Microstates free entropy, {cf.~\cite[\S 2.1]{voi-entropy2}}]\label{micro-def}  Let $n \geq 1$ and $R>0$, let
$(\mathcal{A},\tau)$ be a noncommutative probability space and let $X \in \A_{\sa}^n$ be an $n$-tuple of self-adjoint elements with operator norm $\norm{X}_{\op} \coloneqq \max_j \norm{X_j}_{\op} \leq R$.
\begin{itemize}
\item[(i)]  Let $\C \langle x_1,\dots,x_n \rangle$ be the $*$-algebra of noncommutative polynomials in formal self-adjoint variables $x_1$, \dots, $x_n$.  We define $\Sigma_{n,R}$ as the space of tracial positive linear functionals $\lambda: \C \langle x_1, \dots, x_n \rangle \to \C$ such that for all $i_1$, \dots, $i_\ell \in \{1,\dots,n\}$, we have $|\lambda(x_{i_1} \dots x_{i_\ell})| \leq R^\ell$.  We equip $\Sigma_{n,R}$ with the weak-$*$ topology.
\item[(ii)] We define the \emph{noncommutative law of $X$} in $\Sigma_{n,R}$ to be the linear functional $\lambda_X \in \Sigma_{n,R}$ defined by the formula
$$\lambda_X(p) \coloneqq \tau(p(X)).$$
In particular, in the case where $\A$ is $M_N(\C)$ and $\tau$ is the normalized trace $\tr_N \coloneqq \frac{1}{N} \Tr$, we have $\lambda_Y(p) = \tr_N(p(Y))$ for any $p \in \C\langle x_1,\dots,x_n \rangle$ and any $n$-tuple $(Y_1,\dots,Y_n) \in M_N(\C)_{\sa}^n$ of self-adjoint matrices in $M_N(\C)$. 
\item[(iii)] For an open set $\mathcal{U} \subseteq \Sigma_{n,R}$, we define the \emph{microstate space}\footnote{The condition $Y \in \Gamma_R^{(N)}(\mathcal{U})$ entails that $\lambda_Y \in \Sigma_{n,R}$ and hence $\norm{Y}_{\op} \leq R$.}
\[
\Gamma_R^{(N)}(\mathcal{U}) \coloneqq \{Y \in M_N(\C)_{\sa}^n: \lambda_Y \in \mathcal{U} \}.
\]
\item[(iv)] For $\lambda \in \Sigma_{n,R}$, we define\footnote{The corresponding definition in \cite{voi-entropy2} uses $\frac{1}{2} n \log N$ instead of $n \log N$, but this is due to the use of the un-normalized trace $\Tr$ instead of the normalized trace $\tr_N$ to define the Haar measure $\nu_{M_N(\C)_{\sa}^n}$.}
\[
\chi_R(\lambda) \coloneqq \inf_{\mathcal{U} \ni \lambda} \limsup_{N \to \infty} \frac{1}{N^2} \left( \log \nu_{M_N(\C)_{\sa}^n}(\Gamma_R^{(N)}(\mathcal{U})) + n \log N \right),
\]
where the infimum is taken over all neighborhoods $\mathcal{U}$ of $\lambda$ in $\Sigma_{n,R}$.  
\item[(v)]  We define $\chi(\lambda) \coloneqq \sup_{R'\geq R} \chi_{R'}(\lambda)$.  If $(\A,\tau)$ is a noncommutative probability space and $X \in \A_{\sa}^n$, then we also define $\chi(X) \coloneqq \chi(\lambda_X)$.
\end{itemize}
\end{definition}

The next proposition expresses the microstate entropy $\chi$ in terms of the normalized classical entropies $h^{(N)}$ introduced previously.

\begin{proposition}[Random matrix interpretation of microstates free entropy] \label{prop:variational}
Let $X$ be an $n$-tuple of self-adjoint noncommutative random variables from $(\mathcal{A},\tau)$.  Then $\chi(X)$ is the supremum of
\[
\limsup_{\ell \to \infty} h^{(N_\ell)}(X^{(\ell)})
\]
over all sequences of natural numbers $(N_\ell)_{\ell \in \N}$ tending to $\infty$ and all sequences random variables $(X^{(\ell)})_{\ell \in \N}$ from $M_{N_\ell}(\C)_{\sa}^n$ satisfying the following conditions:
\begin{enumerate}[(1)]
	\item $\lambda_{X^{(N_\ell)}}$ converges in probability to $\lambda_X$.
	\item For some $R > 0$, we have $\limsup_{\ell \to \infty} \norm{X^{(\ell)}}_{\op} \leq R$ in probability, where $\|X\|_{\op}$ denotes the supremum of the operator norms $\|X_i\|_{\op}$ of the components $X_1,\dots,X_n$ of $X$.
	\item There exist some constants $C > 0$ and $K > 0$ such that
\begin{equation}\label{nlarge}
	P(\norm{X^{(\ell)}}_{M_{N_\ell}(\C)_{\sa}^n} \geq C + \delta) \leq e^{-KN_\ell^2 \delta^2} \text{ for all } \delta > 0.
\end{equation}
\end{enumerate}
Furthermore, the supremum (if it is $> -\infty$) is witnessed by random matrices which are uniformly bounded in operator norm and unitarily invariant in distribution.
\end{proposition}

\begin{proof}
First, let $(X^{(\ell)})_{\ell \in \N}$ be a sequence of random matrices as described above, and let $\mu^{(\ell)}$ be the associated probability measure.  Fix $R' > R$.  Let $\mathcal{U}$ be a neighborhood of the noncommutative law of $X$ in $\Sigma_{n,R'}$.  We apply Lemma \ref{prob-standard}(ii) with the partitition $S_j^{(\ell)}, j \geq 0$ of $M_N(\C)_{\sa}^n$ defined by
\begin{align*}
S_0^{(\ell)} &\coloneqq \Gamma_{R'}^{(N_\ell)}(\mathcal{U}) \\
S_1^{(\ell)} &\coloneqq B(0,C+1) \setminus \Gamma_{R'}^{(N_\ell)}(\mathcal{U}) \\
S_j^{(\ell)} &\coloneqq B(0,C+j) \setminus B(0,C+j-1) \text{ for } j \geq 2,
\end{align*}
where $B(0,r)$ denotes the ball of radius $r$ in $M_N(\C)_{\sa}^n$, to obtain
$$
h^{(N_\ell)}(X^{(N_\ell)}) \leq \sum_{j=0}^\infty H^{(\ell)}_j$$
where
$$ H^{(\ell)}_j \coloneqq \mu^{(\ell)}(S_j^{(\ell)}) \left( \frac{1}{N_\ell^2} \log \nu_{M_{N_\ell}(\C)_{\sa}^n}(S_j^{(\ell)}) + n \log N_\ell \right) - \mu^{(\ell)}(S_j^{(\ell)}) \frac{1}{N_\ell^2} \log \mu^{(\ell)}(S_j^{(\ell)}).
$$
We have
$$ H^{(\ell)}_0 = 
\mu^{(\ell)}(\Gamma_{R'}^{(N_\ell)}(\mathcal{U})) \left( \frac{1}{N_\ell^2} \log \nu_{M_{N_\ell}(\C)_{\sa}^n}(\Gamma_{R'}^{(N_\ell)}(\mathcal{U})) + n \log N_\ell \right)
- \mu^{(\ell)}(S_0^{(\ell)}) \frac{1}{N_\ell^2} \log \mu^{(\ell)}(S_0^{(\ell)}).$$
As $\ell \to \infty$, the second term on the right-hand side goes to zero (bounding $- t \log t \leq 1/e$ for any $t > 0$). From Definition \ref{micro-def}, we thus see that for any $\eps>0$ one can find ${\mathcal U}$ for which
$$ \limsup_{\ell \to \infty} H^{(\ell)}_0 \leq \chi(X)+\eps.$$
Next, to estimate $H_1^{(\ell)}$, we observe from a routine application of Stirling's formula (identifying the inner product space $M_{N_\ell}(\C)_{\sa}^n$ with a standard $nN_\ell^2$-dimensional Euclidean space) that
\begin{equation}\label{r-ball}
\frac{1}{N_\ell^2} \log \nu_{M_{N_\ell}(\C)_{\sa}^n}(B(0,r)) = - n \log N_\ell + n \log r + O(n) 
\end{equation}
for any $r>0$. Since $\nu_{M_{N_\ell}(\C)_{\sa}^n}(S_1^{(\ell)}) \leq \nu_{M_{N_\ell}(\C)_{\sa}^n}(B(0,C+1))$ and $\mu^{(\ell)}(S^{(\ell)}_1) \to 0$, we conclude that
$$
\limsup_{\ell \to \infty} H^{(\ell)}_1 \leq 0.$$
For the terms $j \geq 2$, we see from \eqref{r-ball}, \eqref{nlarge} and the fact that $-t \log t$ is increasing for $t \leq 1/e$ that
\begin{align*}
\limsup_{\ell \to \infty} \sum_{j=2}^\infty H^{(\ell)}_j &\leq 
\limsup_{\ell \to \infty} \sum_{j=2}^\infty e^{-KN_\ell^2(j-1)^2} ( n \log j + O(n) + K (j-1)^2 ) \\
&= 0
\end{align*}
Putting all these bounds together, and sending $\eps$ to zero, we conclude that
\[
\limsup_{\ell \to \infty} h^{(N_\ell)}(X^{(\ell)}) \leq \chi(X).
\]
Hence, the supremum of the $\limsup$'s of classical entropies is less than or equal to $\chi(X)$.
	
For opposite inequality, assume without loss of generality that $\chi(X) > -\infty$ since otherwise the inequality is trivial.  Fix $R > \norm{X}_{\op}$.  Let $(\mathcal{U}_\ell)_{\ell \in \N}$ be a sequence of nested neighborhoods of $\lambda_X$ in $\Sigma_{n,R}$ shrinking to $\lambda_X$ as $\ell \to \infty$.  For each $\ell$, choose a number $N_\ell$ such that
\[
\frac{1}{N_\ell^2} \log \nu_{M_{N_\ell}(\C)_{\sa}^n}(\Gamma_R^{(N_\ell)}(\mathcal{U}_\ell)) + n \log N_\ell > \limsup_{N \to \infty} \left( \frac{1}{N^2} \log \nu_{M_N(\C)_{\sa}^n}(\Gamma_{R'}^{(N)}(\mathcal{U})) + n \log N \right) - \frac{1}{\ell}.
\]
We can arrange that $N_{\ell+1} > N_\ell$ and hence $N_\ell \to \infty$.  Define $\mu^{(\ell)}$ to be the uniform measure on $\Gamma_R^{(N_\ell)}(\mathcal{U}_{\ell})$, and let $X^{(\ell)}$ be a random matrix tuple with distribution $\mu^{(\ell)}$.  Then
\[
h^{(N_\ell)}(\mu^{(\ell)}) = \frac{1}{N_\ell^2} \log \nu_{M_{N_\ell}(\C)_{\sa}^n}(\Gamma_R^{(N_\ell)}(\mathcal{U}_{\ell})) + n \log N_\ell.
\]
Hence,
\begin{align*}
\limsup_{N \to \infty} h^{(N_\ell)}(\mu^{(\ell)}) &\geq \limsup_{\ell \to \infty} \left( \limsup_{N \to \infty} \left( \frac{1}{N^2} \log \nu_{M_{N}(\C)_{\sa}^n}(\Gamma_{R'}^{(N)}(\mathcal{U})) + n \log N \right) - \frac{1}{\ell} \right) \\
&= \chi_R(\mu) = \chi(\mu),
\end{align*}
where the last equality follows from \cite[Proposition 2.4]{voi-entropy2}.  Moreover, it is clear that this choice of random matrix models is unitarily invariant and bounded in operator norm.
\end{proof}

\begin{remark} \label{rem:concentration}
Note that assumption (3) is trivially satisfied if $\norm{X^{(\ell)}}_{M_N(\C)_{\sa}^n} \leq C$.  It is also true of any random matrix models which satisfy Herbst's concentration inequality with a suitable normalization depending on the dimension $N_\ell$ (which in turn follows from a normalized log-Sobolev inequality).  In particular, this applies when $X^{(\ell)} = t^{1/2} Z^{(\ell)}$ for a GUE tuple $Z^{(\ell)}$ from $M_{N_\ell}(\C)_{\sa}^d$ and any fixed $t>0$. See \cite{gz} and \cite[\S 4.4.2]{agz}.  Herbst's concentration inequality also implies that (2) holds for some $R$ by \cite[Lemma 11.5.2]{JekelThesis}.
\end{remark}

\begin{remark}
Compare Proposition \ref{prop:variational} to the more explicit connections between microstates free entropy and classical entropy that occur for special random matrix models in \cite{voi-entropy2} and \cite[Proposition 16.1.4]{JekelThesis}.
\end{remark}

\begin{remark}
It was pointed out to us by Ben Hayes (private communication) that a similar idea to Proposition \ref{prop:variational} has already been used in the context of sofic entropy.  Bowen expressed the entropy of algebraic actions of residually finite groups as the supremum of the limits of classical entropies of certain measures on the model spaces (finitary approximations) \cite[Definition 4 and Theorem 4.1]{bowen}.  Similarly, Austin used this approach to define a version of sofic entropy in a more general context \cite{austin}.
\end{remark}

Now we give an analog of Proposition \ref{score-minor}:

\begin{lemma}[Classical score and minors]\label{class-score-minor}
Let $X$ be a random element of $M_N(\C)_{\sa}^n$ with finite classical Fisher information (in particular, the normalized classical score $J^{(N)}(X)$ exists).  Let $1 \leq M \leq N$, and let $\pi^{(N,M)} \colon M_N(\C)_{\sa}^n \to M_M(\C)_{\sa}^n$ be the compression map that sends a tuple $(X_1,\dots,X_n)$ in $M_N(\C)_{\sa}^n$ to the tuple consisting of the upper left $M \times M$ minors of $X_1,\dots,X_n$.  Define the normalized compression
$$ \Pi^{(N,M)} \coloneqq \frac{N^{1/2}}{M^{1/2}} \pi^{(N,M)}.$$
Then $\Pi^{(N,M)}(X)$ has a normalized classical score in $M_M(\C)_{\sa}^n$ given by the formula 
$$ J^{(M)}( \Pi^{(N,M)}(X) ) = \mathbb{E}\left[\Pi^{(N,M)}(J^{(N)}(X)) | \Pi^{(N,M)}(X)\right].$$
\end{lemma}

\begin{proof}
Let $Z^{(N)}$ be a GUE tuple in $M_N(\C)_{\sa}^n$ (classically) independent of $X$.  Then it is easy to see that $\Pi^{(N,M)}(Z^{(N)})$ is a GUE tuple in $M_M(\C)_{\sa}^n$ (classically) independent of $\Pi^{(N,M)}(X)$.  By \eqref{class-score-norm}, it suffices to show that
\begin{align*}
& \frac{d}{d\eps} \mathbb{E} \left\langle f\left(\Pi^{(N,M)}(X) + \eps \Pi^{(N,M)}(Z^{(N)})\right), \Pi^{(N,M)}(Z^{(N)}) \right\rangle_{M_M(\C)_{\sa}^n}|_{\eps = 0} \\
&\quad = \mathbb{E} \left\langle \mathbb{E}\left[\Pi^{(N,M)}(J^{(N)}(X)) | \Pi^{(N,M)}(X)\right], f\left(\Pi^{(N,M)}(X)\right) \right\rangle_{M_M(\C)_{\sa}^n}
\end{align*}
for any smooth $f \colon M_M(\C)_{\sa}^n \to M_M(\C)_{\sa}^n$.  We can remove the conditional expectation on the right-hand side, thus reducing to
\begin{align*}
& \frac{d}{d\eps} \mathbb{E} \left\langle f(\Pi^{(N,M)}(X) + \eps \Pi^{(N,M)}(Z^{(N)})), \Pi^{(N,M)}(Z^{(N)}) \right\rangle_{M_M(\C)_{\sa}^n}|_{\eps = 0} \\
&\quad = \mathbb{E} \left\langle \Pi^{(N,M)}(J^{(N)}(X)),  f(\Pi^{(N,M)}(X)) \right\rangle_{M_M(\C)_{\sa}^n}.
\end{align*}
Embedding $M_M(\C)_{\sa}^n$ into $M_N(\C)_{\sa}^n$ by padding zero entries to the $M \times M$ matrices to create $N \times N$ matrices, this simplifies further to
$$
 \frac{d}{d\eps} \mathbb{E} \left\langle f(\Pi^{(N,M)}(X + \eps Z^{(N)})), Z^{(N)} \right\rangle_{M_N(\C)_{\sa}^n}|_{\eps = 0} = \mathbb{E} \left\langle J^{(N)}(X), f(\Pi^{(N,M)}(X)) \right\rangle_{M_N(\C)_{\sa}^n}.
$$
But this follows from \eqref{class-score-norm}.
\end{proof}

As a consequence we can establish a classical analog of \eqref{ko}, except that there is an error coming from the diagonal elements of the matrix (which will end up going to zero in the limit as $N \to \infty$).

\begin{corollary}[Approximate monotonicity of normalized classical Fisher information]
Let the notation and hypotheses be as in Lemma \ref{class-score-minor}.  If the distribution of $X$ is additionally invariant under unitary conjugation, one has
\begin{equation} \label{eq:approximateKO}
\mathcal{I}^{(M)}\left(\Pi^{(N,M)}(X)\right) \leq \mathcal{I}^{(N)}(X) + \frac{N}{M} \mathbb{E} \norm{\diag(J^{(N)}(X))}_{M_N(\C)_{\sa}^n}^2,
\end{equation}
where $\diag(A) = (\diag(A_1),\dots,\diag(A_n))$ is the orthogonal projection onto the space of diagonal matrices of a tuple $A \in M_N(\C)_{\sa}^n$.
\end{corollary}

\begin{proof}
From Lemma \ref{class-score-minor}, one has
\begin{align*}
\mathcal{I}^{(M)}\left( \Pi^{(N,M)}(X)\right) &= \mathbb{E} \norm{\mathbb{E}\left[\Pi^{(N,M)}(J^{(N)}(X))  | \Pi^{(N,M)}(X)\right]}_{M_M(\C)_{\sa}^n}^2 \\
&\leq \mathbb{E} \norm{\Pi^{(N,M)}(J^{(N)}(X))}_{M_M(\C)_{\sa}^n}^2
\end{align*}
Now let $\sigma$ be a random permutation matrix in $U(N)$, drawn using Haar measure, (classically) independent of $X$.  From the unitary invariance of $X$ we then have
$$ \mathbb{E} \norm{\Pi^{(N,M)}(J^{(N)}(X))}_{M_M(\C)_{\sa}^n}^2 = 
\mathbb{E}_X \mathbb{E}_\sigma \norm{\Pi^{(N,M)}(\sigma J^{(N)}(X) \sigma^{-1})}_{M_M(\C)_{\sa}^n}^2$$
where we use $\mathbb{E}_\sigma$ to denote taking expectation just over $\sigma$, and $\mathbb{E}_X$ to denote taking expectation over the variable $X$ (which is independent of $\sigma$).
For any (deterministic) tuple $A = (A_1,\dots,A_n)$ in $M_N(\C)_{\sa}^n$, with $a_{k,ij}$ denoting the $ij$ entry of $A_k$, direct computation shows that
$$ \| A \|_{M_N(\C)_{\sa}^n}^2 = \frac{1}{N} \sum_{k=1}^n \sum_{1 \leq i,j \leq N} |a_{k,ij}|^2$$
and
$$ \mathbb{E}_\sigma \norm{\Pi^{(N,M)}(\sigma A \sigma^{-1})}_{M_M(\C)_{\sa}^n}^2
= \frac{1}{N} \frac{N(M-1)}{M(N-1)} \sum_{k=1}^n \sum_{1 \leq i,j \leq N: i \neq j} |a_{k,ij}|^2
+ \frac{1}{N} \frac{N}{M} \sum_{k=1}^n \sum_{i=1}^N |a_{k,ii}|^2 $$
and thus (since $N(M-1) \leq M(N-1)$)
$$ \mathbb{E}_\sigma \norm{\Pi^{(N,M)}(\sigma A \sigma^{-1})}_{M_M(\C)_{\sa}^n}^2 \leq \| A \|_{M_N(\C)_{\sa}^n}^2 + \frac{N}{M} \| \mathrm{diag}(A) \|_{M_N(\C)_{\sa}^n}^2.$$
Replacing $A = J^{(N)}(X)$ for each possible value of $X$ and then applying the expectation $\mathbb{E}_X$, we conclude
$$ \mathbb{E} \norm{\Pi^{(N,M)}(J^{(N)}(X))}_{M_M(\C)_{\sa}^n}^2 \leq 
\mathbb{E} \norm{J^{(N)}(X)}_{M_N(\C)_{\sa}^n}^2 + \frac{N}{M} \mathbb{E} \norm{\mathrm{diag}(J^{(N)}(X))}_{M_M(\C)_{\sa}^n}^2$$
and the claim follows.
\end{proof}

We now integrate this to obtain

\begin{lemma}[Approximate monotonicity of normalized classical entropy differences]\label{lem:classicalentropyinequality}
Let $X$ be a random element of $M_N(\C)_{\sa}^n$ with finite variance and unitarily invariant distribution.  Let $Z^{(N)}$ be a GUE tuple in $M_N(\C)_{\sa}^n$ (classically) independent of $X$, and set $X_t \coloneqq X + t^{1/2} Z^{(N)}$.  Let $1 \leq M \leq N$, and let $\Pi^{(N,M)}$ be the normalized compression operator from Lemma \ref{class-score-minor}.  Then for $0 < t_0 < t_1$,
\begin{align*}
h^{(M)}(\Pi^{(N,M)}(X_{t_1})) - h^{(M)}(\Pi^{(N,M)}(X_{t_0}))
 \leq h^{(N)}(X_{t_1}) - h^{(N)}(X_{t_0}) + \frac{n}{2M} \log \frac{t_1}{t_0}.
\end{align*}
\end{lemma}

\begin{proof}  Note that $\Pi^{(N,M)}(X_t) = \Pi^{(N,M)}(X) + t^{1/2} \Pi^{(N,M)}(Z^{(N)})$ and that $\Pi^{(N,M)}(Z^{(N)})$ is a GUE tuple in $M_M(\C)_{\sa}^n$ classically independent of $\Pi^{(N,M)}(X)$.  Hence by two applications of \eqref{debruijn-ident-norm} it suffices to establish the inequality
$$
\mathcal{I}^{(M)}(\Pi^{(N,M)}(X_t))  \leq \mathcal{I}^{(N)}(X_t) + \frac{n}{Mt}
$$
for all $t > 0$. By \eqref{eq:approximateKO}, it suffices to show that
$$
\mathbb{E} \norm{\diag(J^{(N)}(X_t))}_{M_N(\C)_{\sa}^n}^2 \leq \frac{n}{Nt}.
$$
By \eqref{stein-ident-norm}, we have $J^{(N)}(X_t) = \mathbb{E}[ t^{-1/2} Z^{(N)} | X_t]$.  In particular,
\[
\mathbb{E} \norm{\diag(J^{(N)}(X_t))}_{M_N(\C)_{\sa}^n}^2 \leq \mathbb{E} \norm{t^{-1/2} \diag(Z^{(N)})}_{M_N(\C)_{\sa}^n}^2 = \frac{n}{Nt}
\]
as required.
\end{proof}

\begin{proof}[Proof of Theorem \ref{microstates}]
Consider a self-adjoint $n$-tuple $X$ from $(\A,\tau)$ and a freely independent projection $p$ of trace $1/k$ in $\A$.  By enlarging $(\A,\tau)$ if necessary assume it contains a tuple $Z = (Z_1,\dots,Z_n)$ of semicircular variables freely independent of each other and of $X$ and $p$.  Set $X_t \coloneqq X + t^{1/2} Z$ for every $t \geq 0$, and let $\Pi$ be the normalized compression $\Pi \coloneqq k^{1/2} \pi$.

We can assume without loss of generality that $\chi(X) > -\infty$ since otherwise the inequality is trivial.  By Proposition \ref{prop:variational}, there exists a sequence of integers $N_\ell$ tending to $\infty$ and random $N_\ell \times N_\ell$ matrix tuples $X^{(\ell)}$ with $\norm{X^{(\ell)}} \leq R$ and $\lambda_{X^{(\ell)}} \to \mu$ in probability, such that
\[
\chi(X) = \limsup_{\ell \to \infty} h^{(N_\ell)}(X^{(\ell)}).
\]
Let $M_\ell \coloneqq \lceil N_\ell/k \rceil$, so that $M_\ell / N_\ell \to 1/k$ as $\ell \to \infty$.  Let $Z^{(\ell)}$ be a GUE tuple in $M_{N_\ell}(\C)_{\sa}^n$ (classically) independent of $X^{(\ell)}$, and set $X_t^{(\ell)} \coloneqq X^{(\ell)} + t^{1/2} Z^{(\ell)}$ for $t \geq 0$.  Let $\Pi^{(N_\ell,M_\ell)}$ be the normalized compression operator from Lemma \ref{class-score-minor}, and let $P^{(\ell)} \in M_{N_\ell}(\C)$ be the orthogonal projection matrix onto the span of the first $M_\ell$ basis vectors.

It is a standard result in random matrix theory (see, e.g., \cite[Theorem 2.2]{voi-limit-law}, \cite[\S 5.5]{agz}) that $\lambda_{Z^{(\ell)}}$ converges almost surely to $\lambda_Z$.  We also have $\lambda_{X^{(\ell)}} \to \lambda_X$ in probability, and $\lambda_{P^{(\ell)}} \to \lambda_p$.  Because of the independence and unitary invariance of $X^{(\ell)}$ and $Z^{(\ell)}$, Voiculescu's asymptotic freeness theory \cite{voi-limit-law}, \cite{voi-asym-free} implies that $\lambda_{(X^{(\ell)},Z^{(\ell)},P^{(\ell)})} \to \lambda_{(X,Z,p)}$ in probability.  In particular, this implies that $\lambda_{\Pi^{(\ell)}(X_t^{(\ell)})} \to \lambda_{\Pi(X_t)}$ in probability.

Note that for any $t_0 > 0$, the matrix models $\Pi^{(N_\ell,M_\ell)}(X_{t_0}^{(\ell)})$ satisfy the hypotheses of Proposition \ref{prop:variational}.  The tail bound hypothesis (3) follows because $X^{(\ell)}$ is bounded in operator norm and because $Z^{(\ell)}$ satisfies these tail bounds using known concentration inequalities as explained in Remark \ref{rem:concentration}.  Thus, by Proposition \ref{prop:variational},
\[
\chi(\Pi(X_{t_0})) \geq \limsup_{\ell \to \infty} h^{(M_\ell)}\left(\Pi^{(N_\ell,M_\ell)}(X^{(\ell)}_{t_0})\right).
\]
By Lemma \ref{lem:classicalentropyinequality}, for any $0 < t_0 < t_1$, we have
\begin{equation} \label{eq:microstateestimate}
h^{(M_\ell)}\left(\Pi^{(N_\ell,M_\ell)}(X_{t_0}^{(\ell)})\right) - h^{(M_\ell)}\left(\Pi^{(N_\ell,M_\ell)}(X_{t_1}^{(\ell)})\right) \geq 
h^{(N_\ell)}\left(X_{t_0}^{(\ell)}\right) - h^{(N_\ell)}\left(X_{t_1}^{(\ell)}\right) - \frac{n}{2M_\ell} \log \frac{t_1}{t_0}.
\end{equation}
Using Lemma \ref{prob-standard}(iii) and \eqref{gauss-val}, we have
\begin{align*}
h^{(M_\ell)}\left(\Pi^{(N_\ell,M_\ell)}(X_{t_1}^{(\ell)})\right)
&=h^{(M_\ell)}\left(\Pi^{(N_\ell,M_\ell)}(X^{(\ell)}) + t_1^{1/2} \Pi^{(N_\ell,M_\ell)}(Z^{(\ell)})\right)\\
 &\geq h^{(M_\ell)}\left(t_1^{1/2} \Pi^{(N_\ell,M_\ell)}(Z^{(\ell)})\right) \\
&= \frac{n}{2} \log(2 \pi e t_1).
\end{align*}
and similarly
\[
h^{(N_\ell)}\left(X_{t_0}^{(\ell)}\right) \geq h^{(N_\ell)}\left(X^{(\ell)}\right).
\]
Finally, from Lemma \ref{prob-standard}(i) we have
$$
h^{(N_\ell)}\left(X_{t_1}^{(\ell)} \right) \leq \frac{n}{2} \log\left(2\pi e(\Var^{(N_\ell)}(X^{(\ell)})/n + t_1)\right).
$$
Substituting these estimates into \eqref{eq:microstateestimate} and collecting terms, we obtain
$$
h^{(M_\ell)}\left(\Pi^{(N_\ell,M_\ell)}(X_{t_0}^{(\ell)})\right) 
\geq h^{(N_\ell)}(X^{(\ell)}) - n \log\left[2\pi e\left(1 + \frac{\Var^{(N_\ell)}(X^{(\ell)})}{nt_1}\right)\right] - \frac{n}{2M_\ell} \log \frac{t_1}{t_0}.
$$
Taking the $\limsup$ as $\ell \to \infty$, we conclude
\[
\chi(\Pi(X_{t_0})) \geq \chi(X) - n \log\left[2\pi e\left(1 + \frac{\Var(X)}{nt_1}\right)\right],
\]
where $\Var(X) \coloneqq \sum_{j=1}^n \norm{X - \tau(X)}_\tau^2$, and we have observed that $\Var^{(N_\ell)}(X^{(\ell)}) \to \Var(X)$ because $X^{(\ell)}$ is bounded in operator norm and $\lambda_{X^{(\ell)}} \to \lambda_X$ in probability. Taking limits as $t_1 \to \infty$, we obtain
\[
\chi(\Pi(X_{t_0})) \geq \chi(X).
\]
Finally, note that $\Pi(X_{t_0})$ is bounded in operator norm by some constant $R'$ and converges in noncommutative law to $\Pi(X)$; hence, using the upper-semicontinuity of $\chi$ on $\Sigma_{n,R'}$ established in \cite[Proposition 2.6]{voi-entropy2}, we take the limit of the above inequality as $t_0 \to 0$ to conclude
\[
\chi(\Pi(X)) \geq \chi(X)
\]
as required.
\end{proof}

\end{document}